\documentclass[11pt, twoside]{article}

\usepackage{amsmath,amsthm}

\usepackage[utf8]{inputenc}

\usepackage{enumerate}

\usepackage{fancyhdr}
\usepackage{cite}
\usepackage{enumitem}
\usepackage{tikz}
\usetikzlibrary{matrix,calc}

\usepackage{caption}
\usepackage{subcaption}

\usepackage{graphics}

\usepackage{xargs}                      
\usepackage[colorinlistoftodos,prependcaption,textsize=tiny]{todonotes}
\usepackage[normalem]{ulem}

\newcommandx{\pcomment}[2][1=]{\todo[linecolor=red,backgroundcolor=red!25,bordercolor=red,#1]{#2}}
\newcommandx{\kcomment}[2][1=]{\todo[linecolor=blue,backgroundcolor=blue!25,bordercolor=blue,#1]{#2}}

\newif\ifpreprint

\usepackage{hyperref}
\hypersetup{%
    colorlinks=True, 
    citecolor=blue,
    linkcolor=black}

\usepackage[charter]{mathdesign}

\newcommand{\rom}[1]{\uppercase\expandafter{\romannumeral #1\relax}}

\newcommand{\blue}{\color{blue}}

\newcommand{\black}{\color{black}}

  \theoremstyle{definition}
  \newtheorem{theorem}{Theorem}[section]
  \newtheorem{corollary}[theorem]{Corollary}
  
  \newtheorem{lemma}[theorem]{Lemma}
  \newtheorem{definition}[theorem]{Definition}
  \newtheorem{remark}[theorem]{Remark}
   
  \newtheorem{example}[theorem]{Example}

  \newtheorem*{assumption*}{Assumption}

  \numberwithin{equation}{section}
    
  \newcommand{\subjclass}[1]{\bigskip\noindent\emph{2010 Mathematics Subject Classification:}\enspace#1}
  \newcommand{\keywords}[1]{\noindent\emph{Keywords:}\enspace#1}

\newcommand{\VN}{{\mathbf{N}}}

\newcommand{\VR}{{\mathbf{R}}}

\newcommand{\Dsf}{{\mathsf{D}}}

\providecommand{\Co}{{\cal O}}

\providecommand{\FL}{\mathfrak{L}}

\setenumerate[0]{label=(\alph*)}

\newcommand{\eps}{{\varepsilon}}

\newcommand{\ben}{\begin{equation}}
\newcommand{\een}{\end{equation}}

\newcommand{\benn}{\begin{equation*}}
\newcommand{\eenn}{\end{equation*}}

\usepackage[colorinlistoftodos]{todonotes}

\newcommand{\vi}{v^{(1)}}
\newcommand{\vii}{v^{(2)}}

\newcommand{\Ui}{U^{(1)}}
\newcommand{\Uii}{U^{(2)}}

\newcommand{\Ri}{R^{(1)}}

\newcommand{\PPi}{P^{(1)}}

\newcommand{\Pii}{P^{(2)}}
\newcommand{\Piii}{P^{(3)}}
\newcommand{\Pk}{P^{(k)}}
\newcommand{\Pell}{P^{(\ell)}}

\newcommand{\Li}{L^{(1)}}
\newcommand{\Lii}{L^{(2)}}
\newcommand{\Lk}{L^{(k)}}
\newcommand{\Lell}{L^{(\ell)}}

\newcommand{\mi}{m^{(1)}}
\newcommand{\mii}{m^{(2)}}
\newcommand{\mk}{m^{(k)}}
\newcommand{\mell}{m^{(\ell)}}

\newcommand{\nni}{n^{(1)}}
\newcommand{\nii}{n^{(2)}}
\newcommand{\nk}{n^{(k)}}
\newcommand{\nell}{n^{(\ell)}}

\newcommand{\wi}{w^{(1)}}
\newcommand{\wii}{w^{(2)}}
\newcommand{\wk}{w^{(k)}}
\newcommand{\well}{w^{(\ell)}}

\usepackage{authblk}
\usepackage[T1]{fontenc}
\usepackage[utf8]{inputenc}

\author[1]{Phillip Baumann\thanks{E-Mail: phillip.baumann(at)tuwien.ac.at}}
\author[2,3]{Peter Gangl\thanks{E-Mail: gangl(at)math.tugraz.at}}
\author[1]{Kevin Sturm\thanks{E-Mail: kevin.sturm(at)tuwien.ac.at}}

\affil[1]{TU Wien, Wiedner Hauptstr. 8-10,
      1040 Vienna, Austria}
\affil[2]{TU Graz, Steyrergasse 30/III, 
    8010 Graz, Austria}
\affil[3]{FAU Erlangen-N\"urnberg, Cauerstra\ss{}e 11, 91058 Erlangen, Germany}

  \title{Complete topological asymptotic expansion for $L_2$ and $H^1$ tracking-type cost functionals in dimension two and three}

\date{\today}

\begin{document}

\maketitle

\begin{abstract}
In this paper, we study the topological asymptotic expansion of a topology optimisation problem that is constrained by the Poisson equation with the design/shape variable entering through the right hand side. Using an averaged adjoint approach, we give explicit formulas for topological derivatives of arbitrary order for both an $L_2$ and $H^1$ tracking-type cost function 
in both dimension two and three and thereby derive the complete asymptotic expansion. As the asymptotic behaviour 
of the fundamental solution of the Laplacian differs in dimension two and three, also the derivation of the topological expansion 
significantly differs in dimension two and three. The complete expansion for the $H^1$ cost functional directly follows from the analysis of the variation of the state equation. However, the proof of the asymptotics of the $L_2$ tracking-type cost functional is significantly more involved and, surprisingly, the asymptotic behaviour of the bi-harmonic equation plays a crucial role in our proof.

\end{abstract}


  \maketitle
 
\section{Introduction}
The topological derivative was first introduced in \cite{a_ESKOSC_1994a} in the context of a so-called "bubble method", where the iterative placement of bubbles led to a change of the underlying topology and later on mathematically justified in \cite{a_SOZO_1999a,a_GAGUMA_2001a} with an application to linear elasticity in the plane. Furthermore, the topological derivative of shape functionals with various PDE constraints were considered in follow up studies by many authors. For instance, Kirchhoff plates were studied in \cite{a_AMNO_2010a}, electrical impedance tomography in \cite{a_HILA_2008a,a_HILANO_2011a}, Maxwell's equation in \cite{a_MAPOSA_2005a}, Stokes' equation in \cite{a_HAMA_2004a} and elliptic variational inequalities were considered in \cite{a_HILA_2011a,a_LODSNOSO_2017a,a_JAKRRASO_2003a}. Further references and examples can be found in the monograph \cite{b_NOSO_2013a}.


Linear partial differential equations have a wide theoretical foundation and bring along useful tools, such as the fundamental solution, which help studying the asymptotic expansion of the shape functional. Using such tools, many linear problems have been treated in the literature. Yet, also nonlinear problems were investigated. In \cite{a_IGNAROSOSZ_2009a,a_BEMARA_2017a,a_ST_2020a,a_AM_2006b} the first order topological derivative for semilinear problems has been studied. Additionally, quasilinear problems were considered first in \cite{a_AMBO_2017a} and more recently in \cite{a_GAST_2020a,a_AMGA_2019a,a_GAST_2021a}. In fact, in \cite{a_GAST_2020a} a projection trick was used, which does not rely on the use of a fundamental solution.


There are various methods to compute the topological derivative. One of them is the method of Amstutz \cite{phd_AM_2003a}, which, with the help of a perturbed adjoint variable, incorporates the PDE constraint into the expansion of the shape functional and therefore simplifies the analysis. A second method is the one introduced in \cite{a_LODSNOSO_2017a}, where a truncation technique together 
with a suitable Dirichlet-to-Neumann operator and its asymptotic analysis is employed. A third method is the averaged adjoint method, which was introduced in the context of shape optimisation in \cite{a_ST_2015a} and adapted to topology optimisation problems in \cite{a_ST_2020a,a_GAST_2020a}. In contrast to Amstutz' method, the averaged adjoint variable depends on the perturbed state variable and therefore the analysis of the adjoint variable is more challenging. Yet, this method seems to be easily applicable to a wide range of cost functions and the computation for higher order topological derivatives is straight forward as well. Another method was introduced by Delfour in \cite{c_DE_2018b}, which only relies on the unperturbed adjoint variable and therefore does not require an analysis of the adjoint variable. However, it seems to come with the shortcoming that this method is not applicable to certain cost functions.

Topological derivatives are used in numerical algorithms for design optimisation or reconstruction problems. We refer to  \cite{a_AMAN_2006a} where an iterative level-set method is used for design optimisation and to \cite{a_HILA_2008a} where one-shot type methods are considered for electrical impedance tomography.
Higher order topological derivatives can be an important tool to improve the accuracy and performance of numerical algorithms.
In \cite{a_HILANO_2011a,a_BOCO_2017a}, higher order topological derivatives were used to increase the accuracy of one-shot type methods. We also refer to \cite[Chapter 10]{b_NOSOZO_2019a} for more details on a Newton-type method and further applications. Typically, the analytical computation of higher order topological derivatives can be a challenging endeavour and it is even more difficult to find closed formulas for the complete topological expansion for specific cost functionals.


In this paper we present general formulas for higher order topological derivatives in two and three space dimensions, which may pave the way for the computation of higher order formulas for other types of problems as well. This is done by employing a Lagrangian framework based on the averaged adjoint variable, since it enables an iterative and systematic way to compute higher order topological derivatives with no additional effort as long as the complete asymptotic expansion of the averaged adjoint variable is known.
We consider a simple model problem. Let $\Dsf\subset \VR^d$ be 
an open and bounded domain, $\Gamma\subset \partial\Dsf$ with $|\Gamma|>0$, $\Sigma:=\partial \Dsf\setminus\Gamma$ and consider the minimisation problem
\begin{equation}\label{eq:cost_func}
    \text{minimise } \mathcal J(\Omega) := \alpha_1\int_{\Dsf} (u_\Omega-u^\ast)^2\;dx + \alpha_2\int_{\Dsf} |\nabla (u_\Omega-u^\ast)|^2\;dx, \quad \alpha_1,\alpha_2\ge 0,
\end{equation}
subject to $\Omega\subset \Dsf$ and $u_\Omega\in H^1(\Dsf)$, $u_\Omega|_\Gamma=u_D$, such that
\begin{equation}\label{eq:state_Rd}
    \int_{\Dsf} \nabla u_\Omega \cdot \nabla \varphi \;dx = \int_{\Dsf} f_\Omega \varphi \;dx+\int_{\Sigma}u_N\varphi \;dS \quad \text{ for all } \varphi \in H^1_\Gamma(\Dsf),
\end{equation}
where $f_{\Omega}(x):= f_1(x) \chi_\Omega(x) + f_2(x) \chi_{\Omega^c}(x)$ with $f_1,f_2\in L_2(\Dsf)\cap  C^\infty(\Dsf)$, $u_D\in H^{\frac{1}{2}}(\Gamma)$, $u_N\in H^{-\frac{1}{2}}(\Sigma)$ and $u^\ast\in H^1(\Dsf)$. Here, $|\cdot|$ stands for the Euclidean norm, $|\Gamma|:=\text{meas}_{d-1}(\Gamma)$ is the surface measure of $\Gamma$ and $H^1_\Gamma(\Dsf)$ denotes the space of all $H^1(\Dsf)$ functions with vanishing trace on $\Gamma$, that is 
\[
    H^1_\Gamma(\Dsf):=\{\varphi\in H^1(\Dsf)|\; \varphi_{|_{\Gamma}}=0\}.
\] 
For a given inclusion $\omega\subset \VR^d$ with $0\in \omega$ and $x_0\in \Dsf\setminus\overline{\Omega}$, we derive for the cases $\alpha_1=0$ and $\alpha_2=0$ an arbitrary order topological derivative formula for this problem of the form
\[
    \mathcal J(\Omega\cup (x_0 + \eps \omega)) = \mathcal J(\Omega) + \sum_{k=1}^N  \ell_k(\eps) d^k\mathcal J(\Omega)(\omega, x_0) + o(\ell_{N}(\eps)), \quad N\ge 1,
\]
where $x_0 + \eps \omega := \{x_0 + \eps y|\; y \in \omega \}$ denotes the domain perturbation and $\ell_k:\VR^+\to \VR^+$ are continuous functions satisfying 
\[
\lim_{\eps\searrow 0}\ell_k(\eps)=0\quad \text{ and } \quad\lim_{\eps\searrow 0}\frac{\ell_{k+1}(\eps)}{\ell_k(\eps)}=0,\quad \text{ for }k\ge1.
\]
Here, $\VR^+:=(0,\infty)$ denotes the set of positive real numbers.
Furthermore, the real number $d^k\mathcal J(\Omega)(\omega, x_0)$ denotes the $k$-th topological derivative at $\Omega$ evaluated for the inclusion shape $\omega$ and the point of perturbation $x_0$. The explicit form of the functions $\ell_k$ depends for our problem on the space dimension and will significantly differ in dimension $d=2$ vs. $d=3$.

\paragraph{Structure of the paper}
In Section~\ref{sec:2} we study the asymptotic behaviour of \eqref{eq:state_Rd} for the perturbation $\Omega_\eps=\Omega\cup\omega_\eps$ up to arbitrary order. In order to compute the topological derivative for a gradient tracking-type cost function, we then study in Section~\ref{sec:3} the asymptotic behaviour of the associated averaged adjoint equation. This allows us in Section~\ref{sec:4} to derive a closed formula for the topological derivative of the gradient tracking-type cost functional up to arbitrary order, where we also consider some special cases. In Section~\ref{sec:5} we study the asymptotic behaviour of the adjoint state variable associated with the $L_2$ tracking-type cost functional. This leads to a more complex expansion due to appearence of the bi-harmonic equation. Finally, in Section~\ref{sec:6} we derive a general formula for the topological derivative of the $L_2$ tracking-type cost functional of arbitrary order.


\section{Analysis of the state equation}\label{sec:2}

Let $\Omega\subset \Dsf$ open and $\omega\subset \VR^d$ an open, bounded and connected set with $C^1$ boundary $\partial\omega$. We assume that $\omega$ contains the origin $0\in \omega$ and we let $x_0\in \Dsf\setminus\bar{\Omega}$ a fixed spatial point. 
Furthermore, we define the affine transfomation $T_\eps(x) := x_0 + \eps x$ and set $\omega_\eps := T_\eps(\omega)$ for $\eps \ge 0$. In the following we will derive an asymptotic expansion of the perturbed state variable $u_\eps$, which is the unique solution to \eqref{eq:state_Rd} subject to the perturbed domain $\Omega_\eps:=\Omega\cup\omega_\eps$, for $\eps>0$. That is, $u_\eps\in H^1(\Dsf)$ satisfies $u_\eps|_\Gamma=u_D$ and
\begin{equation}\label{eq:state_perturbed}
\int_{\Dsf} \nabla u_\eps \cdot \nabla \varphi \;dx = \int_{\Dsf} f_{\Omega_\eps} \varphi \;dx+\int_{\Sigma}u_N\varphi \;dS \quad \text{ for all } \varphi \in H^1_\Gamma(\Dsf).
\end{equation}
Similarly, the unperturbed state variable $u_0\in H^1(\Dsf)$ satisfies $u_0|_\Gamma=u_D$ and
\begin{equation}\label{eq:state_unperturbed}
\int_{\Dsf} \nabla u_0 \cdot \nabla \varphi \;dx = \int_{\Dsf} f_{\Omega} \varphi \;dx +\int_{\Sigma}u_N\varphi \;dS \quad \text{ for all } \varphi \in H^1_\Gamma(\Dsf).
\end{equation}
\begin{remark}
    Often in topology optimisation one is, additionally to $\Omega_\eps = \omega_\eps \cup \Omega$, also interested in the perturbation $\Omega_\eps:=\Omega\setminus\omega_\eps$ with $\Omega\neq\emptyset$ and $x_0\in\Omega$. The analysis for this perturbation follows the same lines as the one presented in the upcoming sections and the differing perturbation only results in a change of sign of the formula for the topological derivative.
\end{remark}

\paragraph{Notation}
In the following we use the notation $\Dsf_\eps := T^{-1}_\eps(\Dsf)$, $\Gamma_\eps := T^{-1}_\eps(\Gamma)$, as well as $\Sigma_\eps := T^{-1}_\eps(\Sigma)$
and define the weighted $H^1$ norm on $\Dsf_\eps$ by 
\begin{equation}
    \|\varphi\|_\eps :=\|\eps \varphi\|_{L_2(\Dsf_\eps)} + \|\nabla \varphi\|_{L_2(\Dsf_\eps)^d} \quad \text{ for all }\varphi\in H^1(\Dsf_\eps). 
\end{equation}
In what follows, we will also use the convention $\sum_{j=\ell}^ka_j:=0$ whenever $\ell>k$ for any sequence $(a_j)_{j\in \VN}$. Furthermore, we denote the Sobolev-Slobodeckij seminorm of a function $u$ on a bounded domain $\Omega\subset\VR^d$ by
\begin{equation}
   |u|_{H^{\frac{1}{2}}(\Omega)}:=\left(\int_\Omega\int_\Omega\frac{|u(x)-u(y)|^2}{|x-y|^{d+1}}\;dxdy\right)^{\frac{1}{2}}.
\end{equation}
Additionally, we use for a function $f\in C^k(\Omega)$ and $x_0\in\Omega$ the abbreviated notation
\begin{equation}
\nabla^kf(x_0)[x]^k:=\sum_{j_1=1}^d\cdot\cdot\cdot\sum_{j_k=1}^d\frac{\partial^k}{\partial x_{j_1}\cdot\cdot\cdot\partial x_{j_k}}f(x_0)x_{j_1}\cdot\cdot\cdot x_{j_k}.
\end{equation} 
Let us recall rescaling inequalities in Sobolev spaces, which we will need further on; see \cite{a_BAST_2021a}.
\begin{lemma}\label{lem:scaling_inequalities}
    Let $\eps>0$ be fixed. Then there is a constant $C>0$ independent of $\eps$, such that the following holds:
\begin{itemize}
    \item[(a)]  For all $\varphi\in H^1(\Dsf_\eps)$ there holds
\begin{equation}
    \|\eps^{\frac12}\varphi\|_{L_2(\partial\Dsf_\eps)} + |\varphi|_{H^{\frac12}(\partial\Dsf_\eps)}  \le C\|\varphi\|_\eps.
\end{equation}
\item[(b)] Given a smooth connected domain $\Gamma\subset \partial \Dsf$, there is a continuous extension operator $Z_{\Gamma_\eps}:H^{\frac12}(\Gamma_\eps)\to H^1(\Dsf_\eps)$, such that
\begin{equation}
    \|Z_{\Gamma_\eps} (\varphi)\|_{\eps} \le C(\eps^{\frac12}\|\varphi\|_{L_2(\Gamma_\eps)} + |\varphi|_{H^{\frac12}(\Gamma_\eps)}), \quad \text{ for all }\varphi \in H^{1/2}(\Gamma_\eps).
\end{equation}
\item[(c)]  Let $\alpha\in (0,1)$. For every bounded and measurable domain $A\subset \Dsf_\eps$ with $0\in A$, we have
    \begin{equation}
        \|\eps \varphi\|_{L_1(A)} \le \left\{\begin{array}{ll} 
            C\eps^{1-\alpha}\|\varphi\|_\eps   & \text{ for } d=2\\
            C\eps\|\varphi\|_\eps   & \text{ for } d=3
     \end{array}\right. \quad \text{ for all } \varphi\in H^1(\Dsf_\eps).
    \end{equation}
\item[(d)] For every measurable nonempty set $\Gamma\subset \partial \Dsf$, we have
    \begin{equation}
        \|\eps\varphi\|_{L_2(\Dsf_\eps)} \le C \|\nabla \varphi\|_{L_2(\Dsf_\eps)^d} \quad \text{ for all } \varphi \in H^1_{\Gamma_\eps}(\Dsf_\eps).
    \end{equation}
\end{itemize}
\end{lemma}
\begin{proof}
    We refer to \cite[Subsection~3.1]{a_BAST_2021a} for a proof.
\end{proof}

\begin{remark}
    Item (a) of the previous lemma is a scaled version of the trace theorem in $H^1(\Dsf_\eps)$. Item (b) 
    is a scaled version of an extension operator on $H^1_{\Gamma_\eps}(\Dsf_\eps)$. Item (d) is a scaled version of a Friedrich's
    inequality. 
\end{remark}
In the following we derive an asymptotic expansion of $u_\eps$ using the compound layer method (see \cite{b_MANAPL_2012a}, \cite{b_MANAPL_2012_b}). Therefore, we introduce the variations of $u_\eps$, for which we will prove estimates afterwards.

\begin{definition}\label{def:variation_state}
        For almost every $x\in \Dsf_\eps$ we define the first variation of the state $u_\eps$ by 
\begin{equation}
    \Ui_\eps(x) := \left(\frac{u_\eps - u_0}{\eps}\right)\circ T_\eps(x), \quad \eps >0. 
\end{equation}
Furthermore, we define the second variation of the state by
\begin{equation}
    \Uii_\eps := \frac{\Ui_\eps - \Ui - \eps^{d-2}v^{(1)}\circ T_\eps}{\eps}, \quad \eps >0.
\end{equation}
    More generally, we define the $(k+1)$-th variation of $u_\eps$ for $k\ge 2$ and $\eps>0$ by
    \begin{equation}
        U^{(k+1)}_\eps := \begin{cases} \frac{U^{(k)}_\eps - U^{(k)} - \eps^{d-2}v^{(k)}\circ T_\eps-\ln(\eps)b^{(k)}}{\eps} &\quad \text{ for }d=2,\\
\frac{U^{(k)}_\eps - U^{(k)} - \eps^{d-2}v^{(k)}\circ T_\eps}{\eps} &\quad \text{ for }d=3,
\end{cases}
        \end{equation}
        where $\ln$ denotes the natural logarithm. Here $U^{(k)}:\VR^d\to \VR$ are so-called boundary layer correctors, $b^{(k)}\in \VR$ are constants and $v^{(k)}:\Dsf\to \VR$ are 
    regular correctors. The functions $U^{(k)}$ aim to approximate $U_\eps^{(k)}$, however, they 
    introduce an error at the boundary of $\Dsf_\eps$, which is corrected with the help of $v^{(k)}$. 
\end{definition}

\begin{remark}
The difference in the definition of the variation of the state for $k\ge3$ is a consequence of the asymptotic behaviour of the occurring boundary layer correctors. In fact, we will see that for $k\ge 3$ the leading term of $U^{(k)}$ is homogeneous of degree $1$ for $d=3$, whereas in dimension $d=2$ the leading term is given as the natural logarithm.
\end{remark}

For convenience we recall that the fundamental solution of the Laplace operator $(-\Delta)$ in dimension two and three is given for $x\ne 0$ by:
\begin{equation}\label{eq:laplace_fundamental}
    E(x)  := \left\{\begin{array}{cc}
    - \frac{1}{2\pi} \ln(|x|) & \text{ for } d=2, \\
     \frac{1}{4\pi} \frac{1}{|x|} & \text{ for } d = 3.
    \end{array}\right. 
\end{equation}
In the following lemma we introduce the boundary layer correctors $U^{(k)}$ used in the definition of $U^{(k)}_\eps$.
\begin{lemma}\label{lma:def_U}
Define for $k\ge 2$:
    \begin{equation}\label{eq:newton}
        U^{(k)}(x) = \int_\omega E(x-y)F^{(k)}(y)\; dy,
    \end{equation}
where $F^{(k)}(y):= \frac{1}{(k-2)!}\nabla^{k-2}(f_1-f_2)(x_0)[y]^{k-2}$.
Then $U^{(k)}$ satisfies: 
  \begin{equation}\label{eq:uk}
      \int_{\VR^d} \nabla U^{(k)} \cdot \nabla \varphi \;dx = \int_\omega F^{(k)} \varphi \;dx \quad \text{ for all } \varphi \in C^1_c(\VR^d)
\end{equation}
and admits the following asymptotic expansion as $|x| \to \infty$
\begin{equation}\label{eq:expansion_Uk}
   U^{(k)}(x) = R^{(k)}_1(x) + \cdots  + R^{(k)}_N(x) + O(|x|^{-(d-2+N)}),
\end{equation}
where $R^{(k)}_{\ell+1} :\VR^d\to \VR$ are given for $k\ge 2$ and $\ell \ge 0$ by 
\begin{equation}\label{eq:RK}
    R^{(k)}_{\ell+1}(x) = \frac{1}{\ell!} \int_\omega \partial^\ell_t E(x-ty)|_{t=0} F^{(k)}(y)\;dy.
\end{equation}
\end{lemma}
\begin{proof}
    It is clear that $U^{(k)}$ satisfies \eqref{eq:uk}, since $U^{(k)}$ is given by the Newton potential. The asymptotic expansion \eqref{eq:expansion_Uk} follows from a Taylor expansion of $y\mapsto E(x-y)$.
\end{proof}
For $k=1$ we define $\Ui:=0$ and $\Ri_\ell:=0$ for $\ell\ge 1$.
In the following we are going to derive several boundary estimates for $U^{(k)}$ and approximations of these functions. For this purpose we set $b^{(1)}:=0$ and more generally
\begin{equation}\label{eq:bk}
b^{(k)}:=-\frac{1}{2\pi} \int_\omega F^{(k)}(y)\;dy \quad \text{ for }k\ge2.
\end{equation}
In view of $R_1^{(k)}(x) = E(x) \int_\omega F^{(k)}(y)\;dy$ we have by definition
\begin{equation}\label{eq:bk_prop}
 R_1^{(k)}(\eps x)= \begin{cases}
     R_1^{(k)}(x)+\ln(\eps)b^{(k)} &  \text{ for } d=2,\\
     \frac{1}{\eps}R^{(k)}_1(x) & \text{ for } d=3,
 \end{cases}
\end{equation}
which explains the definition of $b^{(k)}$ in dimension two.
\begin{example}\label{ex:example_rk}
We may compute the derivatives \eqref{eq:RK} explicitly for $\ell=1,2,3,4$ and obtain for $d=2$
\begin{align}
    R^{(k)}_1(x) & = -\frac{1}{2\pi}\ln(|x|)\int_\omega F^{(k)}(y)\;dy \\
    R^{(k)}_2(x) & = \frac{1}{2\pi}\frac{x}{|x|^2}\cdot \int_\omega y F^{(k)}(y)\;dy \\
    R^{(k)}_3(x) & = -\frac{1}{4\pi}\frac{1}{|x|^2}\int_\omega \left( |y|^2 - 2 \frac{(x\cdot y)^2}{|x|^2}\right) F^{(k)}(y)\;dy \\
    R^{(k)}_4(x) & = -\frac{1}{12\pi} \frac{1}{|x|^4} \int_\omega \left( 6|y|^2 (x\cdot y) - 8\frac{(x\cdot y)^3}{|x|^2}  \right) F^{(k)}(y)\;dy 
\end{align}
and for $d=3$
\begin{align}
    R^{(k)}_1(x) & = \frac{1}{4\pi}\frac{1}{|x|}\int_\omega F^{(k)}(y)\;dy \\
    R^{(k)}_2(x) & = \frac{1}{4\pi}\frac{x}{|x|^3}\cdot \int_\omega y F^{(k)}(y)\;dy \\
    R^{(k)}_3(x) & = \frac{1}{8\pi}\frac{1}{|x|^3}\int_\omega  \left(-|y|^2 + 3 \frac{(x\cdot y)^2}{|x|^2}\right) F^{(k)}(y)\;dy \\
\end{align}
\end{example}
We will also need remainder estimates for the expansion \eqref{eq:expansion_Uk} of $U^{(k)}$ in various norms:
\begin{lemma}\label{lem:estimate_Uk}
Let $\Gamma_\eps\subset \partial \Dsf_\eps$, $k\ge 2$ and $N\ge 1$. Then there is a constant $C>0$, such that
\begin{itemize}
    \item $\eps^{\frac12}\|U^{(k)} - \sum_{\ell=1}^N R_\ell^{(k)}\|_{L_2(\Gamma_\eps)} \le C \eps^{\frac{d}{2} + N-1}$,
    \item $|U^{(k)} - \sum_{\ell=1}^N R_\ell^{(k)}|_{H^{\frac12}(\Gamma_\eps)} \le C \eps^{\frac{d}{2} + N-1}$,
\item $\|\partial_\nu U^{(k)} - \sum_{\ell=1}^N \partial_\nu R_\ell^{(k)}\|_{L_2(\Gamma_\eps)} \le C \eps^{\frac{d-1}{2} + N}$.
\end{itemize}
\end{lemma}
\begin{proof}
In view of \eqref{eq:expansion_Uk} and \eqref{eq:RK} we have for $x\in \VR^d$:
\[
    \left|U^{(k)}(x) - \sum_{\ell=1}^N R_\ell^{(k)}(x)\right|\le C|x|^{-m}+\mathcal{O}(|x|^{-m-1}),
\] 
with $m=d-2+N$. Thus, an application of \cite[Lemma 3.4]{a_BAST_2021a} yields the result.
\end{proof}

Next we introduce corrector functions which compensate the error introduced by the functions $U^{(k)}$.
\begin{definition}
    We define for $k\ge 1$ the corrector $v^{(k)}\in H^1(\Dsf)$ with $v^{(k)}(x) = -\sum_{j=1}^k R_j^{(k-j+1)}(x-x_0)$ on $\Gamma$ and
    \begin{equation}\label{eq:vk}
        \int_{\Dsf}\nabla v^{(k)} \cdot \nabla \varphi \;dx =\int_{\Sigma}\bigg(\sum_{j=1}^k \partial_\nu R_j^{(k-j+1)}(x-x_0)\bigg) \varphi\; dS \quad \text{ for all } \varphi \in H^1_\Gamma(\Dsf).
    \end{equation}
\end{definition}

\begin{remark}
Note that unique solvability of \eqref{eq:vk} can be shown by the Lemma of Lax-Milgram, and therefore $v^{(k)}$ are well-defined. Furthermore, since $R^{(1)}_\ell=0$, for $\ell\ge 1$, we have $v^{(1)}=0$.\newline By a change of variables, \eqref{eq:vk} can be equivalently written as
$v^{(k)}\circ T_\eps (x) = -\sum_{j=1}^k R_j^{(k-j+1)}(\eps x)$ on $\Gamma_\eps$ and
    \begin{equation}\label{eq:vk_rescaled}
        \int_{\Dsf_\eps}\nabla (\eps^{d-2} v^{(k)}\circ T_\eps) \cdot \nabla \varphi \;dx =\eps^{d-1}\int_{\Sigma_\eps}\bigg(\sum_{j=1}^k \partial_\nu R_j^{(k-j+1)}(\eps x)\bigg) \varphi\; dS \quad \text{ for all } \varphi \in H^1_{\Gamma_\eps}(\Dsf_\eps).
    \end{equation}
\end{remark}

Later on we also need the following auxiliary result.

\begin{lemma}\label{lem:recursion_Ueps}
    Let $\eps >0$ be fixed. We have for all $k\ge2$ and $d=2$:
\begin{equation}
    \Ui_\eps - \eps^{k-1} U^{(k)}_\eps = \sum_{\ell=1}^{k-1} \eps^{\ell-1} \left(U^{(\ell)} + \eps^{d-2} v^{(\ell)}\circ T_\eps+\ln(\eps)b^{(\ell)}\right) \quad \text{ on } \Dsf_\eps,
\end{equation}
with $b^{(\ell)}$ defined as in \eqref{eq:bk}. We have for all $k\ge 2$ and $d=3$:
\begin{equation}
    \Ui_\eps - \eps^{k-1} U^{(k)}_\eps = \sum_{\ell=1}^{k-1} \eps^{\ell-1} \left(U^{(\ell)} + \eps^{d-2} v^{(\ell)}\circ T_\eps\right) \quad \text{ on } \Dsf_\eps.
\end{equation}
\end{lemma}
\begin{proof}
This follows from Definition~\ref{def:variation_state} and a simple induction proof.
\end{proof}
The following lemma will help us to compactly handle the inhomogeneous Dirichlet boundary conditions on $\Gamma_\eps$:
\begin{lemma}\label{lma:aux}
Fix $\eps > 0$. Let $F_\eps:H^1_{\Gamma_\eps}(\Dsf_\eps)\rightarrow \VR$ be a linear and continuous functional with respect to $\|\cdot\|_\eps$ and $g_\eps\in H^{\frac{1}{2}}(\Gamma_\eps)$. Then there exists a unique $V_\eps \in H^1(\Dsf_\eps)$, such that
\begin{equation}\label{eq:aux_vol}
\int_{\Dsf_\eps}\nabla V_\eps\cdot \nabla\varphi\; dx=F_\eps(\varphi)\quad\text{ for all }\varphi\in H^1_{\Gamma_\eps}(\Dsf_\eps),
\end{equation}
\begin{equation}\label{eq:aux_bd}
V_\eps|_{\Gamma_\eps}=g_\eps.
\end{equation}
Furthermore, there exists a constant $C>0$, such that
\begin{equation}\label{eq:aux_ineq}
\|V_\eps\|_\eps\le C(\|F_\eps\|_\eps+\eps^{\frac{1}{2}}\|g_\eps\|_{L_2(\Gamma_\eps)}+|g_\eps|_{H^{\frac{1}{2}}(\Gamma_\eps)}).
\end{equation}
\end{lemma}

\begin{proof}
    We refer to \cite[Lemma~3.8]{a_BAST_2021a} for a proof.
\end{proof}

\begin{corollary}\label{cor:Uk_boundary}
Let $k\ge 2$ and $d=2$. There is a constant $C>0$, such that for all $\eps>0$ small enough:
\begin{align}
    \eps^{\frac12}\|U_\eps^{(k)} - U^{(k)} - \eps^{d-2}v^{(k)}\circ T_\eps-\ln(\eps)b^{(k)}\|_{L_2(\Gamma_\eps)}  & \le  C\eps^{\frac{d}{2}},\\
    |U_\eps^{(k)} - U^{(k)} - \eps^{d-2}v^{(k)}\circ T_\eps-\ln(\eps)b^{(k)}|_{H^{\frac12}(\Gamma_\eps)} & \le C\eps^{\frac{d}{2}}.
\end{align}
Let $k\ge 2$ and $d=3$. There is a constant $C>0$, such that for all $\eps>0$ small enough:
\begin{align}
    \eps^{\frac12}\|U_\eps^{(k)} - U^{(k)} - \eps^{d-2}v^{(k)}\circ T_\eps\|_{L_2(\Gamma_\eps)}  & \le  C\eps^{\frac{d}{2}},\\
    |U_\eps^{(k)} - U^{(k)} - \eps^{d-2}v^{(k)}\circ T_\eps|_{H^{\frac12}(\Gamma_\eps)} & \le C\eps^{\frac{d}{2}}.
\end{align}
\end{corollary}
\begin{proof}
We restrict ourselves to the proof for $d=3$. Let $\eps >0$ be sufficiently small. Using \eqref{eq:bk_prop}, the proof for $d=2$ follows the same lines.
First note that, by definition, $U_\eps^{(k)} - U^{(k)} - \eps^{d-2}v^{(k)}\circ T_\eps = \eps U_\eps^{(k+1)}$ and that, from Lemma \ref{lem:recursion_Ueps}, we have
\begin{equation}\label{eq:recursion_Uk}
    \eps U_\eps^{(k+1)} = \eps^{-(k-1)}U_\eps^{(1)} +  \sum_{\ell =1}^{k} \eps^{\ell-k}(U^{(\ell)} + \eps^{(d-2)}v^{(\ell)}\circ T_\eps) \quad \text{ on }\Dsf_\eps.
\end{equation}
Moreover, since $R_j^{(k-j+1)}(\eps x) = \eps^{-(d-2)}\eps^{-(j-1)}R_j^{(k-j+1)}(x)$ for $1\le j\le k$, we have for $x\in\Gamma_\eps$
\begin{equation}\label{eq:vk_formula}
    v^{(\ell)}\circ T_\eps(x) =  - \eps^{-(d-2)}\sum_{j=1}^\ell \eps^{-(j-1)}R_j^{(\ell-j+1)}(x)
\end{equation}
and thus
\begin{align}
    \sum_{\ell=1}^{k}  \eps^{\ell-k}\eps^{(d-2)}v^{(\ell)}\circ T_\eps(x) & \stackrel{\eqref{eq:vk_formula}}{=} - \eps^{-k}\sum_{\ell=1}^{k} \sum_{j=1}^\ell \eps^{\ell-j+1} R_j^{(\ell-j+1)}(x) \label{eq:reordering_lhs}\\
                                                                     & = - \sum_{\ell=1}^{k} \eps^{\ell-k} \sum_{j=1}^{k-\ell+1} R_j^{(\ell)}(x),\label{eq:reordering_rhs}
\end{align}
where in the last step we reordered the sum as illustrated in Figure \ref{fig:reordering} with $a_{ij}=R^{(j)}_i.$ Therefore, plugging this into \eqref{eq:recursion_Uk} yields
\begin{equation}
    \eps U_\eps^{(k+1)} = \eps^{-(k-1)}U_\eps^{(1)} + \sum_{\ell =1}^{k} \eps^{\ell-k}\left(U^{(\ell)} - \sum_{j=1}^{k-\ell+1} R_j^{(\ell)}\right) \qquad \text{ on } \Dsf_\eps
\end{equation}
and since $\Ui_\eps=0$ on $\Gamma_\eps$, it follows that there is a constant $C>0$, such that
\begin{align}
    \eps^{\frac12}\|\eps U_\eps^{(k+1)}\|_{L_2(\Gamma_\eps)} & \le  \sum_{\ell =1}^{k} \eps^{\ell-k}\underbrace{\eps^{\frac12}\left\|U^{(\ell)} - \sum_{j=1}^{k-\ell+1} R_j^{(\ell)}\right\|_{L_2(\Gamma_\eps)}}_{\le C\eps^{\frac{d}{2}+k-\ell}, \text{ Lemma}~\ref{lem:estimate_Uk}}   
                                                                \le C\eps^{\frac{d}{2}}.
\end{align}
In the same way, using the $H^{\frac12}$ estimate of Lemma~\ref{lem:estimate_Uk}, one can show $|\eps U^{(k+1)}_\eps|_{H^{\frac12}(\Gamma_\eps)} \le C\eps^{\frac{d}{2}}$.
\end{proof}
\begin{figure}
\centering
\begin{subfigure}[b]{0.35\textwidth}
\begin{tikzpicture}
\foreach \x in {1,...,5}
\foreach \y in {1,...,5}
{
\draw (\x,-\y)+(-.42,-.42) rectangle++(.42,.42);
\node (\x,-\y) at (\x,-\y) {$a_{\y\x}$};
}
  \draw[fill=red,rounded corners, opacity=0.4]
    let \p1=($(5,-1)!-3mm!(1,-5)$),
        \p2=($(1,-5)!-3mm!(5,-1)$),
        \p3=($(\p1)!1.5mm!90:(\p2)$),
        \p4=($(\p1)!1.5mm!-90:(\p2)$),
        \p5=($(\p2)!1.5mm!90:(\p1)$),
        \p6=($(\p2)!1.5mm!-90:(\p1)$)
    in
    (\p3) -- (\p4)-- (\p5) -- (\p6) -- cycle;

  \draw[fill=red,rounded corners, opacity=0.4]
    let \p1=($(4,-1)!-3mm!(1,-4)$),
        \p2=($(1,-4)!-3mm!(4,-1)$),
        \p3=($(\p1)!1.5mm!90:(\p2)$),
        \p4=($(\p1)!1.5mm!-90:(\p2)$),
        \p5=($(\p2)!1.5mm!90:(\p1)$),
        \p6=($(\p2)!1.5mm!-90:(\p1)$)
    in
    (\p3) -- (\p4)-- (\p5) -- (\p6) -- cycle;

  \draw[fill=red,rounded corners, opacity=0.4]
    let \p1=($(3,-1)!-3mm!(1,-3)$),
        \p2=($(1,-3)!-3mm!(3,-1)$),
        \p3=($(\p1)!1.5mm!90:(\p2)$),
        \p4=($(\p1)!1.5mm!-90:(\p2)$),
        \p5=($(\p2)!1.5mm!90:(\p1)$),
        \p6=($(\p2)!1.5mm!-90:(\p1)$)
    in
    (\p3) -- (\p4)-- (\p5) -- (\p6) -- cycle;

  \draw[fill=red,rounded corners, opacity=0.4]
    let \p1=($(2,-1)!-3mm!(1,-2)$),
        \p2=($(1,-2)!-3mm!(2,-1)$),
        \p3=($(\p1)!1.5mm!90:(\p2)$),
        \p4=($(\p1)!1.5mm!-90:(\p2)$),
        \p5=($(\p2)!1.5mm!90:(\p1)$),
        \p6=($(\p2)!1.5mm!-90:(\p1)$)
    in
    (\p3) -- (\p4)-- (\p5) -- (\p6) -- cycle;

  \draw[fill=red,rounded corners, opacity=0.4]
    let \p1=($(1.1,-0.9)!-1mm!(0.9,-1.1)$),
        \p2=($(0.9,-1.1)!-1mm!(1.1,-0.9)$),
        \p3=($(\p1)!1.5mm!90:(\p2)$),
        \p4=($(\p1)!1.5mm!-90:(\p2)$),
        \p5=($(\p2)!1.5mm!90:(\p1)$),
        \p6=($(\p2)!1.5mm!-90:(\p1)$)
    in
    (\p3) -- (\p4)-- (\p5) -- (\p6) -- cycle;

\end{tikzpicture}
\end{subfigure}
\begin{subfigure}[b]{0.35\textwidth}
\begin{tikzpicture}
\foreach \x in {1,...,5}
\foreach \y in {1,...,5}
{
\draw (\x,-\y)+(-.42,-.42) rectangle++(.42,.42);
\node (\x,-\y) at (\x,-\y) {$a_{\y\x}$};
}
  \draw[fill=red,rounded corners, opacity=0.4]
    let \p1=($(1,-1)!-3mm!(1,-5)$),
        \p2=($(1,-5)!-3mm!(1,-1)$),
        \p3=($(\p1)!1.5mm!90:(\p2)$),
        \p4=($(\p1)!1.5mm!-90:(\p2)$),
        \p5=($(\p2)!1.5mm!90:(\p1)$),
        \p6=($(\p2)!1.5mm!-90:(\p1)$)
    in
    (\p3) -- (\p4)-- (\p5) -- (\p6) -- cycle;

  \draw[fill=red,rounded corners, opacity=0.4]
    let \p1=($(2,-1)!-3mm!(2,-4)$),
        \p2=($(2,-4)!-3mm!(2,-1)$),
        \p3=($(\p1)!1.5mm!90:(\p2)$),
        \p4=($(\p1)!1.5mm!-90:(\p2)$),
        \p5=($(\p2)!1.5mm!90:(\p1)$),
        \p6=($(\p2)!1.5mm!-90:(\p1)$)
    in
    (\p3) -- (\p4)-- (\p5) -- (\p6) -- cycle;

  \draw[fill=red,rounded corners, opacity=0.4]
    let \p1=($(3,-1)!-3mm!(3,-3)$),
        \p2=($(3,-3)!-3mm!(3,-1)$),
        \p3=($(\p1)!1.5mm!90:(\p2)$),
        \p4=($(\p1)!1.5mm!-90:(\p2)$),
        \p5=($(\p2)!1.5mm!90:(\p1)$),
        \p6=($(\p2)!1.5mm!-90:(\p1)$)
    in
    (\p3) -- (\p4)-- (\p5) -- (\p6) -- cycle;

  \draw[fill=red,rounded corners, opacity=0.4]
    let \p1=($(4,-1)!-3mm!(4,-2)$),
        \p2=($(4,-2)!-3mm!(4,-1)$),
        \p3=($(\p1)!1.5mm!90:(\p2)$),
        \p4=($(\p1)!1.5mm!-90:(\p2)$),
        \p5=($(\p2)!1.5mm!90:(\p1)$),
        \p6=($(\p2)!1.5mm!-90:(\p1)$)
    in
    (\p3) -- (\p4)-- (\p5) -- (\p6) -- cycle;

  \draw[fill=red,rounded corners, opacity=0.4]
    let \p1=($(5,-0.9)!-1mm!(5,-1.1)$),
        \p2=($(5,-1.1)!-1mm!(5,-0.9)$),
        \p3=($(\p1)!1.5mm!90:(\p2)$),
        \p4=($(\p1)!1.5mm!-90:(\p2)$),
        \p5=($(\p2)!1.5mm!90:(\p1)$),
        \p6=($(\p2)!1.5mm!-90:(\p1)$)
    in
    (\p3) -- (\p4)-- (\p5) -- (\p6) -- cycle;

\end{tikzpicture}
\end{subfigure}
\caption{Visualisation of the reordering of the sums in \eqref{eq:reordering_lhs},\eqref{eq:reordering_rhs}}
\label{fig:reordering}
\end{figure}
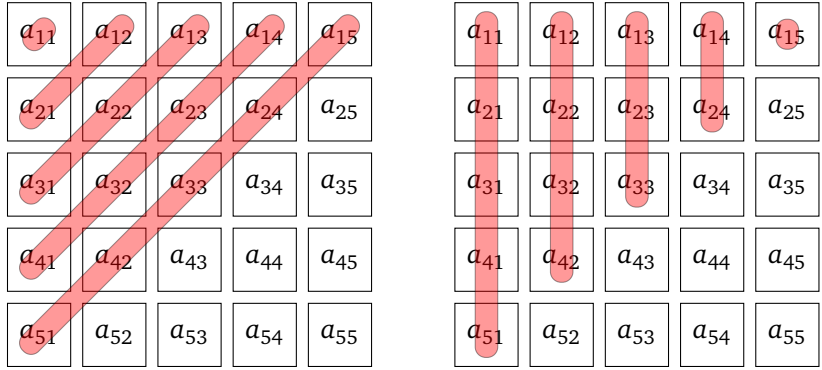

\begin{theorem}\label{thm:asymptotic_Uk}
Let $k\ge 1$ and $\alpha\in (0,1)$. There is a constant $C>0$, such that
\begin{equation}\label{eq:main_2}
    \|U_\eps^{(k)} - U^{(k)} - \eps^{d-2}v^{(k)}\circ T_\eps-\ln(\eps)b^{(k)}\|_\eps \le C \eps^{1-\alpha} \quad \text{ for }d=2,
\end{equation}
\begin{equation}\label{eq:main_3}
    \|U_\eps^{(k)} - U^{(k)} - \eps^{d-2}v^{(k)}\circ T_\eps\|_\eps \le C \eps \quad \text{ for }d=3.
\end{equation}

\end{theorem}
\begin{proof}
We will only prove the estimate for $d=3$. Using \eqref{eq:bk_prop}, the proof for $d=2$ follows the same lines and is therefore left to the reader. Subtracting \eqref{eq:state_perturbed} for $\eps=0$ from \eqref{eq:state_perturbed} with $\eps>0$ we obtain
\begin{equation}
    \int_\Dsf \nabla (u_\eps - u_0) \cdot \nabla \varphi\;dx = \int_\Dsf (f_{\Omega_\eps} - f_\Omega ) \varphi \;dx \quad \text{ for all } \varphi \in H^1_\Gamma(\Dsf)
\end{equation}
and thus, changing variables, we obtain for $\eps >0$:
\begin{equation}\label{eq:U1_eps}
     \int_{\Dsf_\eps} \nabla \Ui_\eps \cdot \nabla \varphi\;dx = \eps \int_\omega (f_1 - f_2)\circ T_\eps \varphi \;dx=:F^{(1)}_\eps(\varphi)
\end{equation}   
for all $\varphi \in H^1_{\Gamma_\eps}(\Dsf_\eps)$. An application of Lemma \ref{lem:scaling_inequalities}, item (c) shows $\|F^{(1)}_\eps\|_\eps \le C\eps$. Now since $\Ui_\eps\in H^1_{\Gamma_\eps}(\Dsf_\eps)$, $\Ui=0$ and $\vi=0$, Lemma \ref{lma:aux} yields the desired estimate \eqref{eq:main_3} for $k=1$. Next we divide by $\eps$ and subtract the equation for  $\Uii$, that is, equation \eqref{eq:uk} for $k=2$ and the rescaled equation for $\eps^{d-2}\vii\circ T_\eps$, that is, equation \eqref{eq:vk_rescaled} for $k=2$ from \eqref{eq:U1_eps} to obtain
\begin{align}\label{eq:Uk_eps0}
    \int_{\Dsf_\eps} \nabla (U_\eps^{(2)}-U^{(2)} - \eps^{d-2}v^{(2)}\circ T_\eps) \cdot \nabla \varphi\;dx =& \int_\omega\bigg((f_1- f_2)\circ T_\eps - (f_1(x_0) - f_2(x_0))\bigg) \varphi \;dx\\
&+\int_{\Sigma_\eps}\left(\partial_\nu \Uii-\eps^{d-1}\partial_\nu R^{(2)}_1(\eps x)\right)\varphi\; dS,
\end{align}
for $\varphi\in H^1_{\Gamma_\eps}(\Dsf_\eps)$. Recalling $\eps U_\eps^{(k+1)} = U_\eps^{(k)}-U^{(k)} - \eps^{d-2}v^{(k)}\circ T_\eps$ and continuing this process we obtain more generally for $k\ge 3$:
\begin{equation}\label{eq:Uk_eps}
    \begin{split}
        \int_{\Dsf_\eps} \nabla (\eps U_\eps^{(k+1)}) \cdot \nabla \varphi\;dx = & \int_\omega \eps^{-(k-2)}\bigg((f_1 - f_2)\circ T_\eps - (f_1(x_0) - f_2(x_0))\bigg) \varphi \;dx \\
                                                                        &-\int_\omega\sum_{\ell=1}^{k-2}  \eps^{-(k-2)+\ell}\frac{\nabla^\ell\big(f_1-f_2\big)(x_0)[x]^\ell}{\ell!} \varphi \;dx\\
&+\int_{\Sigma_\eps}\sum_{\ell=2}^k\eps^{\ell-k} \left(\partial_\nu U^{(\ell)}-\sum_{j=1}^{k-\ell+1}\eps^{d-1+j-1}\partial_\nu R^{(\ell)}_j(\eps x)\right)\varphi\; dS  =: F_\eps^{(k)}(\varphi),
\end{split}
\end{equation}
for $\varphi\in H^1_{\Gamma_\eps}(\Dsf_\eps)$. The Taylor expansion of $(f_1-f_2)\circ T_\eps$ at $\eps =0$ shows for all $\eps$ small enough:
\begin{equation}\label{eq:est_F1}
    \bigg|\int_\omega \eps^{-(k-2)}\bigg((f_1 - f_2)\circ T_\eps - (f_1(x_0) - f_2(x_0))\bigg) \varphi \;dx
                                                                        -\int_\omega\sum_{\ell=1}^{k-2}  \eps^{-(k-2)+\ell}\frac{\nabla^\ell\big(f_1-f_2\big)(x_0)[x]^\ell}{\ell!} \varphi \;dx\bigg|\le C\eps \|\varphi\|_\eps,
\end{equation}
for a constant $C>0$. Furthermore, taking into account that \[\partial_\nu R^{(\ell)}_j(\eps x)=\eps^{-(d-2+j)}\partial_\nu R^{(\ell)}_j(x)\quad \text{ for }\ell,j\ge 1,\]
it follows from Lemma \ref{lem:estimate_Uk} and Hölder's inequality that
\begin{align}\label{eq:est_F2}
\begin{split}
\bigg|\int_{\Sigma_\eps}\sum_{\ell=2}^k\eps^{\ell-k} \left(\partial_\nu U^{(\ell)}-\sum_{j=1}^{k-\ell+1}\eps^{d-1+j-1}\partial_\nu R^{(\ell)}_j(\eps x)\right)\varphi\; dS \bigg|\le&\sum_{\ell=2}^k\eps^{\ell-k}\left\Vert\partial_\nu U^{(\ell)}-\sum_{j=1}^{k-\ell+1}\partial_\nu R^{(\ell)}_j(x)\right\Vert_{L_2(\Sigma_\eps)}\|\varphi\|_{L_2(\Sigma_\eps)}\\
\le& C \eps^{\frac{d+1}{2}}\|\varphi\|_{L_2(\Sigma_\eps)}\\
\le& C \eps^{\frac{d}{2}}\|\varphi\|_\eps,
\end{split}
\end{align}
for a constant $C>0$, where in the last step we used the continuity of the trace operator (see Lemma \ref{lem:scaling_inequalities}, item (a)). Combining \eqref{eq:est_F1} and \eqref{eq:est_F2} we get $\|F_\eps^{(k)}\|_\eps \le C \eps$ for a constant $C>0$ and $k\ge2$. Additionally, we deduce from Corollary \ref{cor:Uk_boundary} that \[\eps^{\frac12}\|\eps U_\eps^{(k+1)}\|_{L_2(\Gamma_\eps)}+|\eps U_\eps^{(k+1)}|_{H^{\frac12}(\Gamma_\eps)}\le C\eps^{\frac{d}{2}},\]
for a positive constant $C>0$. Thus, Lemma \ref{lma:aux} yields \eqref{eq:main_3} and therefore finishes the proof.
\end{proof}


\blue
\paragraph{Specialisation to spherical inclusion $\omega=B_1(0)$}
We now assume that the inclusion $\omega=B_1(0)$ is the unit ball in $\VR^d$ centered at the origin and $f_1,f_2\in \VR$. In this case we see that $U^{(2)}\ne0$ and $U^{(k)}=0$ for all $k\ge 3$. Moreover, we have
\begin{equation}
    U^{(2)}(x) = (f_1-f_2) \int_{B_1(0)} E(x-y)\;dy.  
\end{equation}
Now we note that $y\mapsto E(x-y)$ is harmonic for all $x\in \VR^d\setminus \overline{B_1(0)}$ and thus by the mean value theorem for harmonic functions 
\begin{equation}
U^{(2)}(x) = (f_1-f_2) |B_1(0)| E(x).
\end{equation}
This means that $R^{(2)}_1 (x) = (f_1-f_2) |B_1(0)| E(x)$ and $R^{(2)}_\ell(x) =0$ for all $\ell \ge 2$.  Moreover, as mentioned before, we have  $U^{(k)}=0$ for all $k\ge 3$, which implies that $R^{(k)}_\ell(x)=0$ for all $\ell\ge 1$ and $k\ge 3$. Therefore
\begin{equation}
    \sum_{j=1}^k R_j^{(k-j+1)}(x) = \begin{cases}
    R_1^{(2)}(x) & \text{ for } k=2, \\
    0 & \text{ for } k\ge 3.
    \end{cases}
\end{equation}
This implies in particular that $v^{(k)}=0$ for all $k\ge 3$. Hence it follows from \eqref{thm:asymptotic_Uk} that
\begin{align}
    U_\eps^{(2)} - U^{(2)} - v^{(2)}\circ T_\eps-\ln(\eps)b^{(2)} = 0 \quad \text{ for } d=2, \\
    U_\eps^{(2)} - U^{(2)} - \eps v^{(2)}\circ T_\eps = 0 \quad \text{ for } d=3,
\end{align}
or equivalently
\begin{align}
    u_\eps  = u_0 + \eps^2\left( U^{(2)}\circ T_\eps^{-1} + v^{(2)} + \ln(\eps) b^{(2)} \right) \quad \text{ for } d=2, \label{eq:exact_1} \\
    u_\eps  = u_0 + \eps^3\left( \eps^{-1}U^{(2)}\circ T_\eps^{-1} +  v^{(2)} \right) \quad \text{ for } d=3 \label{eq:exact_2}.
\end{align}
 We finally show that $U^{(2)}$ can be explicitly  computed. In fact, $U^{(2)}$ solves  
\begin{equation}
    -\Delta U^{(2)} = (f_1-f_2)\chi_{B_1(0)} \quad \text{ in }\VR^d. 
\end{equation}
But since $U^{(2)}(x) = (f_1-f_2) |B_1(0)| E(x)$ for all $x\in \VR^d\setminus \overline{B_1(0)}$ and since $U^{(2)}$ is continuous, it follows that  $U^{(2)}(x) = (f_1-f_2) |B_1(0)| E(x)$ on $\partial B_1(0)$ and thus $U^{(2)}$ must be a solution to the inhomogeneous Dirichlet problem: find $U^{(2)}\in H^1(B_1(0))$, such that $U^{(2)}(x) = (f_1-f_2) |B_1(0)| E(x)$ on $\partial B_1(0)$ and 
\begin{equation}
    -\Delta U^{(2)} = (f_1-f_2) \quad \text{ in } B_1(0). 
\end{equation}
It is readily checked using polar coordinates that the solution is given for $d=2$ by 
\begin{equation}
    U^{(2)}(x) = \begin{cases}
        -(f_1-f_2) \frac{1}{4}(\|x\|^2-1) \quad & \text{ for } x\in B_1(0), \\
        -(f_1-f_2) \frac{1}{2} \ln(\|x\|) \quad & \text{ for } x\in \VR^2\setminus\overline{B_1(0)},
    \end{cases}
\end{equation}
and in dimension $d=3$ using spherical coordinates leads to:
\begin{equation}
    U^{(2)}(x) = \begin{cases}
    -(f_1-f_2) \frac{1}{6}(\|x\|^2-3) \quad & \text{ for } x\in B_1(0), \\
       (f_1-f_2) \frac{1}{3} \frac{1}{\|x\|} \quad & \text{ for } x\in \VR^3\setminus\overline{B_1(0)}.
    \end{cases}
\end{equation}
Recall that the function $v^{(2)}$ was defined in \eqref{eq:vk} and is given by 
\begin{align}
   && -\Delta v^{(2)}    & = 0 &\quad  \text{ in } \Dsf,&&  \\
   && v^{(2)} & = - R^{(2)}_1(x-x_0) & \quad \text{ on } \Gamma,&& \\
   && -\partial_\nu v^{(2)} & = R^{(2)}_1(x-x_0)&  \quad \text{ on }\Sigma.&&
\end{align}
Note that according to Example~\ref{ex:example_rk}, and the fact that $F^{(2)}(x)=f_1-f_2$ and $|B_1(x_0)|$ is equal to $\pi$ for $d=2$ and equal to $\frac{4\pi}{3}$ for $d=3$, we have
\begin{equation}
R^{(2)}_1(x) =   \begin{cases}
      -\frac{f_1-f_2}{2}\ln(|x|) & \text{ for } d=2, \\
      \frac{f_1-f_2}{3|x|} & \text{ for } d=3.
    \end{cases}
\end{equation}
For a general domain $\Dsf$ its solution cannot be explicitly computed.  However, we know $U^{(2)}$ explicitly and thus can write the expansion of $u_\eps$ as follows
\begin{align}
    u_\eps & = u_0 + \eps^2\begin{cases}
        \left(  - \frac{f_1-f_2}{4}(\eps^{-2} \|x-x_0\|^2 - 1)  + v^{(2)}  - \frac{f_1-f_2}{2}\ln(\eps)\right) & \quad \text{ for } x\in B_\eps(x_0),\\
        \left( -\frac{f_1-f_2}{2}\ln(\|x-x_0\|) + v^{(2)} \right)  & \quad \text{ for } x\in \Dsf\setminus\overline{B_\eps(x_0)},  
    \end{cases}
\end{align}
and in dimension $d=3$:
\begin{align}
    u_\eps  = u_0 + \eps^3\begin{cases}
        \left( -(f_1-f_2)\frac16 (\eps^{-3}\|x-x_0\|^2 -3) + v^{(2)}   \right) \quad & \text{ for } x\in B_\eps(x_0),\\
        \left(  \frac{f_1-f_2}{3}\frac{1}{\|x-x_0\|} + v^{(2)} \right) \quad & \text{ for } x\in \Dsf\setminus\overline{B_\eps(x_0)}. 
\end{cases}
\end{align}
 Note that the $\ln(\eps)$ term in $d=2$ disappears outside of $B_\eps(0)$ and that indeed $u_\eps - u_0 =0$ on $\partial \Dsf$.
\black

\section{Analysis of the averaged adjoint equation for the $H^1$ tracking-type cost function}\label{sec:3}
Since the analysis of the averaged adjoint variable for the $L_2$ tracking-type cost functional differs significantly from the analysis of the $H^1$ tracking-type cost functional, we split the cost functional $J$ defined in \eqref{eq:cost_func} into two parts and treat each one seperately. Thus, in this section we derive the asympotics of the averaged adjoint state $p_\eps$ for the $H^1$ tracking-type part of \eqref{eq:cost_func}. The $L_2$ tracking-type part of \eqref{eq:cost_func} is treated in Section \ref{sec:5}. 

\subsection{$H^1$ tracking-type cost function and averaged adjoint}\label{sec:3.1}
We consider the cost function
\begin{equation}\label{eq:cost}
    \mathcal J_2(\Omega) := \alpha_2\int_\Dsf |\nabla (u_\Omega-u^\ast)|^2 \;dx,
\end{equation}
where $\alpha_2\ge0$ and $u_\Omega\in H^1(\Dsf)$ satisfies $u_\Omega|_\Gamma=u_D$ and
\begin{equation}
    \int_\Dsf \nabla u_\Omega \cdot \nabla \varphi \;dx=\int_\Dsf f_\Omega \varphi\; dx+\int_\Sigma u_N \varphi\; dS \quad\text{ for all }\varphi\in H^1_\Gamma(\Dsf). 
\end{equation}
The associated Lagrangian is given by 
\begin{equation}
    \FL(\eps,\varphi,\psi) :=\alpha_2\int_\Dsf |\nabla (\varphi-u^\ast)|^2 \;dx+ \int_\Dsf \nabla \varphi\cdot\nabla \psi  - f_{\Omega_\eps} \psi\;dx- \int_\Sigma u_N\psi\;dS, \quad \varphi,\psi\in H^1_\Gamma(\Dsf).
\end{equation}
 Now, the averaged adjoint equation reads: find $p_\eps\in H^1_\Gamma(\Dsf)$, such that  
\begin{align}\label{eq:averaged_adjoint_abstract}
    \int_0^1 \partial_u \FL(\eps, s u_\eps + (1-s) u_0, p_\eps)(\varphi) \; ds = 0 \quad  \mbox{ for all } \varphi \in H^1_\Gamma(\Dsf).
\end{align}\label{eq:averaged_adjoint_concrete}
Or explicitly, evaluating the $ds$-integral, the perturbed averaged adjoint equation reads: find $p_\eps \in H^1_\Gamma(\Dsf)$, such that
\begin{equation}\label{eq:adj_per_H1}
    \int_\Dsf \nabla \varphi \cdot \nabla p_\eps \;dx = -\alpha_2 \int_\Dsf \nabla (u_\eps + u_0 - 2u^\ast)\cdot \nabla \varphi\;dx \quad \text{ for all } \varphi \in H^1_\Gamma(\Dsf). 
\end{equation}
By setting $\eps=0$ we get the unperturbed averaged adjoint equation: find $p_0 \in H^1_\Gamma(\Dsf)$, such that
\begin{equation}\label{eq:adj_unper_H1}
    \int_\Dsf \nabla \varphi \cdot \nabla p_0 \;dx = - 2\alpha_2\int_\Dsf \nabla (u_0 - u^\ast)\cdot \nabla \varphi\;dx \quad \text{ for all } \varphi \in H^1_\Gamma(\Dsf). 
\end{equation}

\subsection{Asymptotic analysis of the averaged adjoint}
Analogously to the definition of the variation of the state, we define the variation of the averaged adjoint state $P^{(k)}_\eps$, $k\ge1$ by replacing the correctors $U^{(k)}, v^{(k)}$ and constants $b^{(k)}$ in Definition \ref{def:variation_state} by correctors $P^{(k)}, w^{(k)}$ and constants $c^{(k)}$ adapted to the differing right hand side in \eqref{eq:adj_per_H1}.
Hence, we can deduce the following analogue to Lemma \ref{lem:recursion_Ueps};
\begin{lemma}\label{lem:recursion_Peps_H1}
    Let $\eps >0$ be fixed. We have for all $k\ge2$ and $d\in \{2,3\}$:
\begin{equation}
    \PPi_\eps - \eps^{k-1} P^{(k)}_\eps = \sum_{\ell=1}^{k-1} \eps^{\ell-1} \left(P^{(\ell)} + \eps^{d-2} w^{(\ell)}\circ T_\eps+ \delta_{2,d}\ln(\eps)c^{(\ell)}\right) \quad \text{ on } \Dsf_\eps,
\end{equation}
where $\delta_{2,d}$ denotes the Kronecker delta function that satisfies $\delta_{2,d}=1$ if $d=2$ and $\delta_{2,d}=0$ else.
\end{lemma}
By subtracting \eqref{eq:adj_per_H1} from \eqref{eq:adj_unper_H1} changing variables with $T_\eps$ and dividing the results by $\eps>0$, we see that the first variation of the averaged adjoint state satisfies
\begin{equation}\label{eq:var_adj_H1}
\int_{\Dsf_\eps}\nabla \PPi_\eps\cdot\nabla\varphi\; dx=-\alpha_2\int_{\Dsf_\eps}\nabla\Ui_\eps\cdot\nabla\varphi\; dx\quad\text{ for all }\varphi\in H^1_{\Gamma_\eps}(\Dsf_\eps).
\end{equation}
It follows that $ \PPi_\eps = - \alpha_2\Ui_\eps$ and therefore the asymptotic behaviour of $\PPi_\eps$ 
is up to a factor identical to the one of  $\Ui_\eps$. We summarise this result in the following theorem. 

\begin{theorem}\label{thm:asymptotic_Pk}
For $\ell\ge1$ let $P^{(\ell)}:=-\alpha_2 U^{(\ell)}$, $w^{(\ell)}:=-\alpha_2 v^{(\ell)}$ and $c^{(\ell)}:=-\alpha_2 b^{(\ell)}$, with $U^{(\ell)}$ defined in Lemma \ref{lma:def_U}, $v^{(\ell)}$ defined in \eqref{eq:vk} and $b^{(\ell)}$ defined in \eqref{eq:bk}. Additionally, let $k\ge 1$ and $\alpha\in (0,1)$. Then there is a constant $C>0$, such that for all $\eps>0$ sufficiently small:
\begin{equation}\label{eq:main_adj_2}
    \|P_\eps^{(k)} - P^{(k)} - \eps^{d-2}w^{(k)}\circ T_\eps-\ln(\eps)c^{(k)}\|_\eps \le C \eps^{1-\alpha} \quad \text{ for }d=2,
\end{equation}
\begin{equation}\label{eq:main_adj_3}
    \|P_\eps^{(k)} - P^{(k)} - \eps^{d-2}w^{(k)}\circ T_\eps\|_\eps \le C \eps \quad \text{ for }d=3.
\end{equation}
\end{theorem}
\begin{proof}
    Since $ \PPi_\eps = - \alpha_2\Ui_\eps$ the result follows from Theorem~\ref{thm:asymptotic_Uk}.
\end{proof}
\section{Complete topological expansion - $H^1$ tracking-type} \label{sec:4}
In this section we compute the $n$-th topological derivative of the $H^1$ tracking-type part of the cost function defined in \eqref{eq:cost}. That is, we are deriving an asymptotic expansion of the form
\begin{equation}\label{eq:topo_def_H1}
\mathcal J_2(\Omega_\eps)=\mathcal J_2(\Omega)+\sum_{k=1}^n \ell_k(\eps) d^{k}\mathcal J_2(\Omega)(\omega, x_0)+o(\ell_n(\eps)),
\end{equation}
with $\mathcal J_2(\Omega)$ defined as in \eqref{eq:cost}.
Here $d^k\mathcal J_2(\Omega)(\omega, x_0)$ denotes the $k$-th topological derivative with respect to the initial domain $\Omega$ for the perturbation shape $\omega$ at the point $x_0$ and $\ell_k:\VR^+\rightarrow \VR^+$ are continuous functions satisfying 
\[
\lim_{\eps\searrow 0}\ell_k(\eps)=0\quad \text{ and } \quad\lim_{\eps\searrow 0}\frac{\ell_{k+1}(\eps)}{\ell_k(\eps)}=0,\quad \text{ for }k\ge1.
\]
As we will see, the logarithmic term $\ln(\eps)c^{(k)}$ in the asymptotic expansion of the adjoint state variable in $d=2$ leads to a differing topological derivative compared to dimension $d=3$. Thus, we will distinguish between both scenarios and derive a general formula of the topological derivative for both cases separately.

The following lemma helps us to compute the product of two finite sums in view of the asymptotic behaviour with respect to $\eps>0$.
\begin{lemma}\label{lma:cauchy_landau}
For $N\ge0$, $x\in\omega$ and $\eps>0$ small let \[f_\eps(x):=\sum_{n=0}^N \eps^n a_n(x)+\mathcal{O}(\eps^{N+1};x),\quad g_\eps(x):=\sum_{n=0}^N \eps^n b_n(x)+\mathcal{O}(\eps^{N+1};x),\]where $a_n,b_n: \omega\rightarrow \VR$, $n\ge0$ are functions independent of $\eps$. Then
\begin{equation}
f_\eps(x)g_\eps(x)=\sum_{n=0}^N\eps^n\sum_{j=0}^n a_j(x)b_{n-j}(x)+\mathcal{O}(\eps^{N+1};x).
\end{equation}
\end{lemma}
\begin{proof}
This can be shown by setting $a_n(x)=b_n(x)=0$ for $n>N$ and computing the Cauchy product \[\left(\sum_{n=0}^\infty \eps^n a_n(x)\right)\left(\sum_{n=0}^\infty \eps^n b_n(x)\right)=\sum_{n=0}^\infty\eps^n\sum_{j=0}^na_j(x)b_{n-j}(x).\]
\end{proof}
\subsection{General formula for higher order topological derivatives in $d=2$}
In this section we treat the dimension $d=2$. Recall that we considered the case of $x_0\in\Dsf\setminus\bar{\Omega}$. We have the following result:
\begin{theorem}\label{thm:deriv_2_H1}
    Let $\ell_1(\eps):=|\omega_\eps|$, $\ell_{2n}(\eps)=\eps^n \ln(\eps)|\omega_\eps|$ and $\ell_{2n+1}(\eps)=\eps^n |\omega_\eps|$, for $n\ge1$.
    The topological derivative of $\mathcal J_2$ at $x_0\in\Dsf\setminus\bar{\Omega}$ and $\omega\subset \VR^2$ with $0\in \omega$ in dimension $d=2$ is given by
\begin{equation}\label{eq:topo_formula_1}
d^1 \mathcal J_2(\Omega) (\omega, x_0)=\big((f_2-f_1)p_0\big)(x_0),
\end{equation}
    \begin{equation}\label{eq:topo_formula_2}
            d^{2n}\mathcal J_2(\Omega) (\omega, x_0)  =   \frac{1}{|\omega|}\left( \sum_{j=0}^{n-2}\frac{1}{j!}\int_\omega \nabla^j\big(f_2-f_1\big)(x_0)[x]^j c^{(n-j)}\; dx\right),
   \end{equation}
    \begin{align}\label{eq:topo_formula_3}
        \begin{split}
            d^{2n+1}\mathcal J_2(\Omega) (\omega, x_0)  = &    \frac{1}{|\omega|}\frac{1}{n!}\int_\omega \nabla^n\big((f_2-f_1)p_0\big)(x_0)[x]^n \;dx\\
&+\frac{1}{|\omega|}\left(\sum_{j=0}^{n-2}\frac{1}{j!}\int_\omega  \nabla^j\big(f_2-f_1\big)(x_0)[x]^j P^{(n-j)}(x)\; dx\right)\\
          & + \frac{1}{|\omega|}\left( \sum_{j=0}^{n-2}\frac{1}{j!}\int_\omega \nabla^j \big((f_2-f_1)w^{(n-j)}\big)(x_0)[x]^j\; dx\right),
    \end{split}
   \end{align}
for $n\ge1$, where $w^{(\ell)},P^{(\ell)}$ and $c^{(\ell)}$, $\ell\ge 1$ are defined in Theorem \ref{thm:asymptotic_Pk}.

\end{theorem} 
\begin{proof}
Recall the Lagrangian introduced in Section \ref{sec:3.1}. Let $\eps\ge 0$. We first 
observe by testing \eqref{eq:averaged_adjoint_abstract} with $\varphi =u_\eps -u_0$ that
\begin{equation}\label{eq:J_L}
    \mathcal J_{2}(\Omega_\eps)  = \FL(\eps,u_\eps,p_\eps) = \FL(\eps,u_0,p_\eps)
\end{equation}
so that the cost function can be written only in terms of the averaged adjoint variable. Therefore, we have
\begin{equation}\label{eq:cost_diff}
    \mathcal J_{2}(\Omega_\eps) - \mathcal J(\Omega)  = \FL(\eps,u_0,p_\eps) - \FL(\eps,u_0,p_0) + \FL(\eps,u_0,p_0) - \FL(0,u_0,p_0).
\end{equation}
Using Theorem~\ref{thm:asymptotic_Pk}, we now derive an expansion for both differences on the right hand side.
\paragraph{Expansion of $\FL(\eps,u_0,p_0) - \FL(0,u_0,p_0)$:} 
We have
\begin{align}\label{eq:diff_Lagrangian_averaged1_H1}
    \begin{split}
    \FL(\eps,u_0,p_0) - \FL(0,u_0,p_0) & =-\int_{\omega_\eps} (f_1-f_2)p_0\; dx\\
                                       & = \eps^d \int_\omega \underbrace{\big((f_2 -f_1)p_0\big)\circ T_\eps}_{=:\hat{p}_0\circ T_\eps}\;dx,
\end{split}
\end{align}
where we used a change of variables in the last step. Hence, a Taylor expansion of $\eps \mapsto \hat{p}_0\circ T_\eps$ in $\eps=0$ yields
\begin{equation}\label{eq:taylor_expansion_hatf_H1}
    \hat{p}_0\circ T_\eps(x) = \hat{p}_0(x_0) + \sum_{k=1}^N \eps^k \frac{\nabla^k \hat{p}_0(x_0)[x]^k}{k!} + \Co(\eps^{N+1}).
\end{equation}
Now plugging this Taylor expansion \eqref{eq:taylor_expansion_hatf_H1} into \eqref{eq:diff_Lagrangian_averaged1_H1}, we get
\begin{align}\label{eq:deriv_part1_H1}
    \frac{\FL(\eps,u_0,p_0) - \FL(0,u_0,p_0)}{|\omega_\eps|} &  =  \hat{p}_0(x_0) +  \sum_{k=1}^N \eps^k \frac{1}{|\omega|}\frac{1}{k!}\int_\omega \nabla^k\hat{p}_0(x_0)[x]^k \;dx + \Co(\eps^{N+1}), 
\end{align}
where $|\omega_\eps|$ denotes the volume of $\omega_\eps$.
Thus, we have discovered one part of the $n$-th order topological derivative.


\paragraph{Expansion of $\FL(\eps,u_0,p_\eps) - \FL(\eps,u_0,p_0)$:}

We proceed in several steps. First, we compute 
\begin{align}\label{eq:diff_Lagrangian_averaged}
    \begin{split}
    \FL(\eps,u_0,p_\eps) - \FL(\eps,u_0,p_0)  =& \int_\Dsf \nabla u_0\cdot \nabla(p_\eps - p_0) - f_{\Omega_\eps}(p_\eps-p_0)\;dx - \int_\Sigma u_N (p_\eps - p_0) \; dS \\
    =& \underbrace{\int_\Dsf \nabla u_0\cdot \nabla(p_\eps - p_0) - f_{\Omega}(p_\eps-p_0)\;dx - \int_\Sigma u_N (p_\eps - p_0) \; dS }_{=0, \text{ in view of }\eqref{eq:state_unperturbed}}  - \int_\Dsf (f_{\Omega_\eps}-f_{\Omega})(p_\eps-p_0)\;dx \\
    =& \eps^d \int_\omega \left(f_2 - f_1\right)\circ T_\eps \eps\PPi_\eps\;dx,
\end{split}
\end{align}
where in the last step we used the change of variable $T_\eps$ and the definition $\eps\PPi_\eps = (p_\eps-p_0)\circ T_\eps$. 
We now substitute $\eps\PPi_\eps$ by the recursion formula of Lemma~\ref{lem:recursion_Peps_H1} and obtain
\begin{align}
    \int_\omega \left (f_2 - f_1\right) \circ T_\eps \eps\PPi_\eps\;dx =&  \sum_{n=1}^{N} \int_\omega \left(f_2 - f_1\right) \circ T_\eps\eps^n P^{(n)}\;dx \label{eq:first_term}\\
                                                               & + \sum_{n=1}^{N}\int_\omega \left(f_2 - f_1\right) \circ T_\eps \eps^n w^{(n)}\circ T_\eps \;dx\label{eq:second_term} \\
                                                              & + \sum_{n=1}^{N}\int_\omega \left(f_2 - f_1\right) \circ T_\eps \eps^n \ln(\eps)c^{(n)} \;dx\label{eq:third_term} \\
                                                               & + \int_\omega \left(f_2 - f_1\right) \circ T_\eps \eps^{N+1}P_\eps^{(N+1)} \;dx. \label{eq:fourth_term}
\end{align}
Now we can expand all four terms:    
\begin{itemize}
    \item First term \eqref{eq:first_term}: We use Taylor's expansion to write:
        \begin{align}\label{eq:taylor_f}
            \left(f_2-f_1\right)\circ T_\eps(x)  = \sum_{j=0}^{N} \eps^j a_j(x) + \Co(\eps^{N+1};x), \qquad a_j(x) := \frac{\nabla^j\big(f_2-f_1\big)(x_0)[x]^j}{j!}
        \end{align}
        For the proof we set $P^{(0)}:=0$. Then, by Lemma \ref{lma:cauchy_landau} we have
        \begin{equation}
            \begin{split}
                \left( f_2-f_1\right)\circ T_\eps(x)\left( \sum_{n=1}^{N} \eps^n P^{(n)}(x) \right)  =
                                                                                                              \sum_{n=0}^{N}  \eps^n \left(\sum_{j=0}^n a_j(x) P^{(n-j)}(x)\right) + \Co(\eps^{N+1};x)
        \end{split}
        \end{equation}
and further, taking into account that $P^{(0)}=P^{(1)}=0$,
        \begin{equation}\label{eq:remainder_dJk_1}
            \sum_{n=1}^{N} \int_\omega \left(f_2-f_1\right)\circ T_\eps\eps^n P^{(n)}\;dx = \sum_{n=2}^{N}  \eps^n \int_\omega \left(\sum_{j=0}^{n-2} a_j(x) P^{(n-j)}(x)\right)\;dx + \Co(\eps^{N+1}).
        \end{equation}
\item Second term \eqref{eq:second_term}: 
Again we use Taylor's formula to expand the functions $\eps \mapsto \hat{w}^{(n)}\circ T_\eps$ at $\eps=0$ with $\hat{w}^{(n)}:=(f_2-f_1)w^{(n)}$, $n\ge1$ to deduce
    \begin{equation}\label{eq:taylor_hatg}
        \hat{w}^{(n)}(x)  = \sum_{j=0}^{N} \eps^j b_j^{(n)}(x) + \mathcal O(\eps^{N+1};x), \qquad b_j^{(n)}(x) := \frac{\nabla^j \hat w^{(n)}(x_0)[x]^j}{j!}.
    \end{equation}
    Hence, a similar computation as in the first bullet point and an application of Lemma \ref{lma:cauchy_landau} yield
    \begin{equation}\label{eq:remainder_dJk_2}
    \begin{split}
        \sum_{n=1}^{N}\int_\omega \eps^n \left((f_2-f_1)w^{(n)}\right)\circ T_\eps \;dx =& \sum_{n=1}^{N} \eps^n \left( \sum_{j=0}^{N} \eps^j \int_\omega b_j^{(n)}(x)\;dx \right) + \Co(\eps^{N+1})\\
=&\sum_{n=2}^N \eps^n \left( \sum_{j=0}^{n-2}\int_\omega b_j^{(n-j)}(x)\; dx\right)+\Co(\eps^{N+1}),
    \end{split}
    \end{equation}
where we took into account $w^{(1)}=0$ and therefore $b^{(1)}_j=0$ for $j\ge0$ as well.
\item Third term \eqref{eq:third_term}: In view of the terms $a_j(x)$, $j\ge0$ introduced above in \eqref{eq:taylor_f}, we have
    \begin{equation}\label{eq:remainder_dJk_3}
    \begin{split}
        \sum_{n=1}^{N}\int_\omega \eps^n\ln(\eps) \left(f_2-f_1\right)\circ T_\eps c^{(n)} \;dx =& \left(\sum_{j=0}^{N} \eps^j\ln(\eps) \int_\omega a_j(x)\;dx\right) \left( \sum_{n=1}^{N} \eps^n c^{(n)} \right) + o(\eps^{N+1})\\
=&\sum_{n=2}^N \eps^n\ln(\eps) \left( \sum_{j=0}^{n-2}\int_\omega a_j(x)c^{(n-j)}\; dx\right)+o(\eps^{N}),
    \end{split}
    \end{equation}
where we took into account that $c^{(1)}=0$.
\item Fourth term \eqref{eq:fourth_term}: Applying Lemma~\ref{lem:scaling_inequalities}, item (c), to the last term and using the asymptotics derived in Theorem~\ref{thm:asymptotic_Pk} gives
\begin{align}\label{eq:remainder_dJk_4}
\begin{split}
    \left|\int_\omega \left(f_1 - f_2\right)\circ T_\eps \eps^{N+1} P_\eps^{(N+1)}\;dx\right| \le& C\eps^{N-\alpha}\|\eps P_\eps^{(N+1)}\|_\eps \le C \eps^{N+1-2\alpha},
\end{split}
\end{align}
for a constant $C>0$ and $\alpha\in (0,1)$ sufficiently small.
\end{itemize}
Combining \eqref{eq:remainder_dJk_1} - \eqref{eq:remainder_dJk_4} leaves us with the expansion
\begin{align}
     \frac{\FL(\eps,u_0,p_\eps) - \FL(\eps,u_0,p_0)}{|\omega_\eps|}  = & \frac{1}{|\omega|}\sum_{k=2}^{N}  \eps^k  \left(\sum_{j=0}^{k-2}\int_\omega a_j(x) P^{(k-j)}(x)\; dx\right)\\
          & + \frac{1}{|\omega|}\sum_{k=2}^N \eps^k \left( \sum_{j=0}^{k-2}\int_\omega b_j^{(k-j)}(x)\; dx\right) \\
&+\frac{1}{|\omega|}\sum_{k=2}^N \eps^k\ln(\eps) \left( \sum_{j=0}^{k-2}\int_\omega a_j(x)c^{(k-j)}\; dx\right)
                                                                   + o(\eps^{N}).
\end{align}
From this formula, together with \eqref{eq:deriv_part1_H1}, we see that the $n$-th topological derivative is given by \eqref{eq:topo_formula_1}-\eqref{eq:topo_formula_3}.

\end{proof}
As shown in Theorem~\ref{thm:asymptotic_Pk}, we have $\PPi_\eps = - \alpha_2\Ui_\eps$. Thus, we deduce the following result.
\begin{corollary}\label{cor:formula_2_H1}
The first five topological derivatives in dimension $d=2$ read as follows:
\begin{align}
    d^1\mathcal J_2(\Omega)(\omega, x_0)=&\big((f_2-f_1)p_0\big)(x_0), && \ell_1(\eps) = |\omega| \eps^2,\\
    d^2\mathcal J_2(\Omega)(\omega, x_0)=&0, && \ell_2(\eps) = |\omega|\eps^2(\eps\ln(\eps)), \\
    d^3\mathcal J_2(\Omega)(\omega, x_0)=&\frac{1}{|\omega|}\int_\omega\nabla\big((f_2-f_1)p_0\big)(x_0)[x]\; dx, && \ell_3(\eps) = |\omega| \eps^3,\\
d^4\mathcal J_2(\Omega)(\omega, x_0)=&\alpha_2[f_1-f_2](x_0)b^{(2)}, && \ell_4(\eps) = |\omega| \eps^3(\eps\ln(\eps)),\\
d^5\mathcal J_2(\Omega)(\omega, x_0)=&\frac{1}{2|\omega|}\int_\omega\nabla^2\big((f_2-f_1)p_0\big)(x_0)[x]^2\; dx && \ell_5(\eps) = |\omega|\eps^4,\\
&+\frac{\alpha_2\big(f_1-f_2\big)(x_0)}{|\omega|}\int_\omega\Uii\; dx\\
&+\alpha_2\big((f_1-f_2)\vii\big)(x_0).
\end{align}
\end{corollary}
\subsubsection{Special cases}
In this section we consider some special cases for our input data in dimension $d=2$ and discuss how this influences the topological derivative.\newline
At first, let the inhomogeneity be piecewise constant, that is, assume $f_1,f_2\in \VR$. Thus, it follows from \eqref{eq:uk} that $U^{(k)}=0$ for $k\neq 2$. Additionally, we have $P^{(k)}=0$ and $b^{(k)}=c^{(k)}=0$, for $k\neq 2$. Since $a_j(x)$ denotes the $j$-th term of the Taylor's expansion of $\left(f_2-f_1\right)\circ T_\eps$, we further deduce that $a_j(x)=0$ for $j>0$. These observations yield the following result:
\begin{corollary}
    Assume that $f_1,f_2\in\VR$. Let $\ell_1(\eps):=|\omega_\eps|$, $\ell_{2n}(\eps):=\eps^n\ln(\eps)|\omega_\eps|$ and $\ell_{2n+1}(\eps)=\eps^{n} |\omega_\eps|$, for $n\ge1$.
    The topological derivative of $\mathcal J_2$ at $x_0\in\Dsf\setminus\bar{\Omega}$ and $\omega\subset \VR^d$ with $0\in \omega$ in dimension $d=2$ is given by
\begin{equation}\label{eq:topo_formula_6}
d^1 \mathcal J_2 (\Omega) (\omega, x_0)=(f_2-f_1)p_0(x_0),\quad d^2 \mathcal J_2 (\Omega) (\omega, x_0)=0
\end{equation}
    \begin{equation}\label{eq:topo_formula_7}
           d^3 \mathcal J_2 (\Omega) (\omega, x_0)=\frac{f_2-f_1}{|\omega|}\int_\omega\nabla p_0(x_0)[x]\; dx,\quad d^{4}\mathcal J_2(\Omega) (\omega, x_0) =   (f_2-f_1)c^{(2)},
   \end{equation}
    \begin{align}\label{eq:topo_formula_8}
        \begin{split}
            d^{5}\mathcal J_2(\Omega) (\omega, x_0)  = &    \frac{1}{|\omega|}\frac{f_2-f_1}{2}\int_\omega \nabla^2 p_0(x_0)[x]^2 \;dx\\
&+\frac{f_2-f_1}{|\omega|}\int_\omega  P^{(2)}(x)\; dx\\
          & + (f_2-f_1)w^{(2)}(x_0),
    \end{split}
   \end{align}
\begin{equation}\label{eq:topo_formula_9a}
d^{2n} \mathcal J_2 (\Omega) (\omega, x_0)=0
\end{equation}
    \begin{align}\label{eq:topo_formula_9}
        \begin{split}
            d^{2n+1}\mathcal J_2(\Omega) (\omega, x_0)  = &    \frac{f_2-f_1}{|\omega|}\frac{1}{n!}\int_\omega \nabla^np_0(x_0)[x]^n \;dx\\
          & + \frac{f_2-f_1}{|\omega|}\left( \sum_{j=0}^{n-2}\frac{1}{j!}\int_\omega \nabla^j w^{(n-j)}(x_0)[x]^j\; dx\right),
    \end{split}
   \end{align}
for $n\ge3$, where $w^{(\ell)},P^{(\ell)}$ and $c^{(\ell)}$, $\ell\ge 1$ are defined in Theorem~\ref{thm:asymptotic_Pk}.
\end{corollary}
Next we consider a symmetric inclusion $\omega\subset \VR^2$. To be precise, we assume that for each point $(x,y)\in \omega$ we have $(-x,y)\in\omega$ and $(x,-y)\in \omega$. From this, one readily checks that
\begin{equation}\label{eq:prop_special_2}
\int_\omega\nabla^k h(x_0) [x]^k\; dx=0\quad \text{ for }k \text{ odd}
\end{equation}
for a sufficiently smooth function $h$.
As a result, the odd numbered boundary layer correctors and the corresponding logarithmic terms vanish. That is,
\[U^{(k)}=P^{(k)}=0,\quad \text{ for }k \text{ odd},\]
\[b^{(k)}=c^{(k)}=0,\quad \text{ for }k \text{ odd}.\]
Additionally, since we derived the topological derivative by expanding $p_0$, $f_2-f_1$ and $w^{(k)}$, $k\ge 2$ with the help of Taylor's expansion, some terms in the general formula can be skipped by the same argument. These considerations yield the following corollary:
\begin{corollary}
    Assume that $f_1,f_2\in\VR$ and the perturbation shape $\omega$ is symmetric. Let $\ell_1(\eps):=|\omega_\eps|$, $\ell_{2n}(\eps)=\eps^n \ln(\eps)|\omega_\eps|$ and $\ell_{2n+1}(\eps)=\eps^{n} |\omega_\eps|$, for $n\ge1$.
    The topological derivative of $\mathcal J_2$ at $x_0\in\Dsf\setminus\bar{\Omega}$ and $\omega\subset \VR^2$ with $0\in \omega$ in dimension $d=2$ is given by
\begin{equation}\label{eq:topo_formula_10}
d^1 \mathcal J_2 (\Omega) (\omega, x_0)=(f_2-f_1)p_0(x_0),\quad d^2 \mathcal J_2 (\Omega) (\omega, x_0)=0
\end{equation}
    \begin{equation}\label{eq:topo_formula_11}
           d^3 \mathcal J_2 (\Omega) (\omega, x_0)=0,\quad d^{4}\mathcal J_2(\Omega) (\omega, x_0) =   (f_2-f_1)c^{(2)},
   \end{equation}
    \begin{align}\label{eq:topo_formula_12}
        \begin{split}
            d^{5}\mathcal J_2(\Omega) (\omega, x_0)  = &    \frac{1}{|\omega|}\frac{f_2-f_1}{2}\int_\omega \nabla^2 p_0(x_0)[x]^2 \;dx\\
&+\frac{f_2-f_1}{|\omega|}\int_\omega  P^{(2)}(x)\; dx\\
          & + (f_2-f_1)w^{(2)}(x_0),
    \end{split}
   \end{align}
\begin{equation}\label{eq:topo_formula_13}
d^{2n} \mathcal J_2 (\Omega) (\omega, x_0)=0
\end{equation}
    \begin{align}\label{eq:topo_formula_14}
        \begin{split}
            d^{2n+1}\mathcal J_2(\Omega) (\omega, x_0)  = &    \frac{f_2-f_1}{|\omega|}\frac{1}{n!}\int_\omega \nabla^np_0(x_0)[x]^n \;dx\\
          & + \frac{f_2-f_1}{|\omega|}\left( \sum_{j=0}^{\frac{n-2}{2}}\frac{1}{(2j)!}\int_\omega \nabla^{2j} w^{(n-2j)}(x_0)[x]^{2j}\; dx\right)\quad \text{ for }n \text{ even},
    \end{split}
   \end{align}
    \begin{align}\label{eq:topo_formula_15}
        \begin{split}
            d^{2n+1}\mathcal J_2(\Omega) (\omega, x_0)  = \frac{f_2-f_1}{|\omega|}\left( \sum_{j=0}^{\frac{n-3}{2}}\frac{1}{(2j)!}\int_\omega \nabla^{2j} w^{(n-2j)}(x_0)[x]^{2j}\; dx\right)\quad \text{ for }n \text{ odd},
    \end{split}
   \end{align}
for $n\ge3$, where $w^{(\ell)},P^{(\ell)}$ and $c^{(\ell)}$, $\ell\ge 1$ are defined in Theorem~\ref{thm:asymptotic_Pk}.
\end{corollary}
Let us finish this section with computing the first five topological derivatives for the unit ball.
\begin{corollary}\label{cor:formula_2_H1_ball}
For the inclusion $\omega:=B_1(0)$ the unit ball in $\VR^2$ centered at the origin and $f_1,f_2\in\VR$, the topological derivatives in dimension $d=2$ read:
\begin{align*}
    d^1\mathcal J_2(\Omega)(\omega, x_0)=&(f_2-f_1)p_0(x_0), && \ell_1(\eps) = |\omega|\eps^2,\\
    d^2\mathcal J_2(\Omega)(\omega, x_0)=&0, && \ell_2(\eps) = |\omega|\eps^2(\eps\ln(\eps)),\\
d^3\mathcal J_2(\Omega)(\omega, x_0)=&0, && \ell_3(\eps) = |\omega|\eps^ 3,\\
d^4\mathcal J_2(\Omega)(\omega, x_0)=&-\frac{\alpha_2(f_1-f_2)^2}{2}, && \ell_4(\eps) = |\omega|\eps^3(\eps\ln(\eps)),\\
d^5\mathcal J_2(\Omega)(\omega, x_0)=&\frac{(f_2-f_1)}{2\pi}\int_\omega\nabla^2p_0(x_0)[x]^2\; dx, && \ell_5(\eps) = |\omega|\eps^4,\\
&+\frac{\alpha_2(f_1-f_2)}{\pi}\int_\omega\Uii\; dx\\
&+\alpha_2(f_1-f_2)\vii(x_0).
\end{align*}
\end{corollary}


\subsection{General formula for higher order topological derivatives in $d=3$}
Similarly to the previous section one can derive the following result regarding the topological derivative of the $H^1$ tracking-type part of the cost functional in dimension $d=3$.


\begin{theorem}
    Let $\ell_1(\eps):=|\omega_\eps|$ and $\ell_{n}(\eps)=\eps^{n-1} |\omega_\eps|$, for $n\ge2$.
    The topological derivative of $\mathcal J_2$ at $x_0\in\Dsf\setminus\bar{\Omega}$ and $\omega\subset \VR^3$ with $0\in \omega$ in dimension $d=3$ is given by
\begin{equation}\label{eq:topo_formula_4}
d^1 \mathcal J _2(\Omega) (\omega, x_0)=\big((f_2-f_1)p_0\big)(x_0),\quad d^2 \mathcal J_2 (\Omega) (\omega, x_0)=\frac{1}{|\omega|}\int_\omega\nabla \big((f_2-f_1)p_0\big)(x_0)[x]\; dx
\end{equation}
    \begin{align}\label{eq:topo_formula_5}
        \begin{split}
            d^{n+1}\mathcal J_2(\Omega) (\omega, x_0) = &    \frac{1}{|\omega|}\frac{1}{n!}\int_\omega \nabla^n \big((f_2-f_1)p_0\big)(x_0)[x]^n \;dx\\
&+\frac{1}{|\omega|}\left(\sum_{j=0}^{n-2}\frac{1}{j!}\int_\omega  \nabla^j \big(f_2-f_1\big)(x_0)[x]^j P^{(n-j)}(x)\; dx\right)\\
          & + \frac{1}{|\omega|}\left( \sum_{j=0}^{n-3}\frac{1}{j!}\int_\omega \nabla^j \big((f_2-f_1)w^{(n-1-j)}\big)(x_0)[x]^j\; dx\right).
    \end{split}
   \end{align}
for $n\ge 2$, where $w^{(\ell)},P^{(\ell)}$ and $c^{(\ell)}$, $\ell\ge 1$ are defined in Theorem~\ref{thm:asymptotic_Pk}.

\end{theorem}
\begin{proof}
This can be shown similarly to the proof in $d=2$. The main difference is the shift in the third term of \eqref{eq:topo_formula_5}, which is a result of the factor $\eps^{d-2}=\eps^1$ in Lemma~\ref{lem:recursion_Peps_H1}. Additionally, the absence of the logarithmic terms leads to a clearer representation of the general formula of the topological derivative compared to the two dimensional case.
\end{proof}
Similarly to Corollary \ref{cor:formula_2_H1}, we get the following result.
\begin{corollary}\label{ex:formula_3_H1}
The first three topological derivatives in terms of the correctors of $\Ui_\eps$ in dimension $d=3$ are given as
\begin{align}
d^1\mathcal J_2(\Omega)(\omega, x_0)=&\big((f_2-f_1)p_0\big)(x_0),&& \ell_1(\eps) = |\omega|\eps^3,\\
d^2\mathcal J_2(\Omega)(\omega, x_0)=&\frac{1}{|\omega|}\int_\omega\nabla\big((f_2-f_1)p_0\big)(x_0)[x]\; dx,&& \ell_2(\eps) = |\omega|\eps^4,\\
d^3\mathcal J_2(\Omega)(\omega, x_0)=&\frac{1}{2|\omega|}\int_\omega\nabla^2\big((f_2-f_1)p_0\big)(x_0)[x]^2\; dx, && \ell_3(\eps) = |\omega|\eps^5, \\
&+\frac{\alpha_2\big(f_1-f_2\big)(x_0)}{|\omega|}\int_\omega\Uii\; dx.
\end{align}
\end{corollary}

\section{Analysis of the averaged adjoint equation - $L_2$ tracking-type}\label{sec:5}
In this section we derive the asympotics of the averaged adjoint state $p_\eps$ for the $L_2$ tracking-type part of \eqref{eq:cost_func}. The analysis differs significantly from the 
$H^1$ tracking-type cost function and involves the fundamental solution of the bi-harmonic equation. Since the bi-harmonic equation is not homogeneous in dimension $d=2$, the analysis 
in this case is more complicated and requires the introduction of several regular corrector equations.

\subsection{$L_2$ tracking-type cost function and averaged adjoint}\label{sec:3.1b}
We consider the cost function
\begin{equation}\label{eq:cost_b}
    \mathcal J_1(\Omega) := \alpha_1\int_\Dsf (u_\Omega-u^\ast)^2 \;dx,
\end{equation}
where $\alpha_1\ge0$ and $u_\Omega\in H^1(\Dsf)$ satisfies $u_\Omega|_\Gamma=u_D$ and
\begin{equation}
    \int_\Dsf \nabla u_\Omega \cdot \nabla \varphi \;dx=\int_\Dsf f_\Omega \varphi\; dx+\int_\Sigma u_N \varphi\; dS \quad\text{ for all }\varphi\in H^1_\Gamma(\Dsf). 
\end{equation}
The associated Lagrangian is given by 
\begin{equation}\label{eq:lagrangian_L2}
    \FL(\eps,\varphi,\psi) :=\alpha_1\int_\Dsf (\varphi-u^\ast)^2 \;dx+ \int_\Dsf \nabla \varphi\cdot\nabla \psi  - f_{\Omega_\eps} \psi\;dx - \int_\Sigma u_N \psi \; dS, \quad \varphi,\psi\in H^1_\Gamma(\Dsf).
\end{equation}
 Now, the averaged adjoint equation reads: find $p_\eps\in H^1_\Gamma(\Dsf)$, such that  
\begin{align}\label{eq:averaged_adjoint_abstract_b}
    \int_0^1 \partial_u \FL(\eps, s u_\eps + (1-s) u_0, p_\eps)(\varphi) \; ds = 0 \quad  \mbox{ for all } \varphi \in H^1_\Gamma(\Dsf).
\end{align}\label{eq:averaged_adjoint_concrete_b}
Or explicitly, evaluating the $ds$-integral, the perturbed averaged adjoint equation reads: find $p_\eps \in H^1_\Gamma(\Dsf)$, such that
\begin{equation}\label{eq:adj_per_b}
    \int_\Dsf \nabla \varphi \cdot \nabla p_\eps \;dx = -\alpha_1 \int_\Dsf (u_\eps + u_0 - 2u^\ast)\varphi\;dx \quad \text{ for all } \varphi \in H^1_\Gamma(\Dsf). 
\end{equation}
By setting $\eps=0$ we get the unperturbed averaged adjoint equation: find $p_0 \in H^1_\Gamma(\Dsf)$, such that
\begin{equation}\label{eq:adj_unper}
    \int_\Dsf \nabla \varphi \cdot \nabla p_0 \;dx = - \alpha_1\int_\Dsf 2(u_0 - u^\ast)\varphi\;dx \quad \text{ for all } \varphi \in H^1_\Gamma(\Dsf). 
\end{equation}
\subsection{Analysis of the averaged adjoint}
Due to the differing boundary layer correctors occurring in the analysis of the averaged adjoint state variable associated with this specific cost functional, some modifications of the variation of the averaged adjoint state compared to the direct state have to be considered. Thus, we will introduce them in the following:
\begin{definition}\label{def:variation_adjoint_state}
        For almost every $x\in \Dsf_\eps$ we define the first variation of the averaged adjoint state $p_\eps$ by 
\begin{equation}
    \PPi_\eps(x) := \left(\frac{p_\eps - p_0}{\eps}\right)\circ T_\eps(x), \quad \eps >0. 
\end{equation}
Furthermore, we define the second variation of the averaged adjoint state by
\begin{equation}
    \Pii_\eps := \frac{\PPi_\eps}{\eps}, \quad \eps >0,
\end{equation}
and the third variation of the averaged adjoint state for $\eps>0$ by
\begin{equation}
    \Piii_\eps := \begin{cases}\frac{\Pii_\eps-\eps^2\Pii-\eps^{d-2}\Lii_\eps\circ T_\eps-\eps^{d-2}\mii\circ T_\eps-\ln(\eps)\nii\circ T_\eps}{\eps}\quad &\text{ for }d=2,\\
\frac{\Pii_\eps-\eps^2\Pii-\eps^{d-2}\Lii_\eps\circ T_\eps-\eps^{d-2}\mii\circ T_\eps}{\eps}\quad &\text{ for }d=3.\end{cases}
\end{equation}
    More generally, we define the $(k+1)$-th variation of $p_\eps$ for $k\ge 3$ and $\eps>0$ by
    \begin{equation}\label{eq:variation_adjoint}
        P^{(k+1)}_\eps := \begin{cases}\frac{\Pk_\eps - \eps^2\Pk - \eps^{d-2}\Lk_\eps\circ T_\eps - \eps^{d}\wk\circ T_\eps-\eps^{d-2}\mk\circ T_\eps-\ln(\eps)\nk\circ T_\eps}{\eps}\quad &\text{ for }d=2,\\
\frac{\Pk_\eps - \eps^2\Pk - \eps^{d-2}\Lk_\eps\circ T_\eps - \eps^{d}\wk\circ T_\eps-\eps^{d-2}\mk\circ T_\eps}{\eps}\quad &\text{ for }d=3.\end{cases}
    \end{equation}
Here, $\Pk$, $\mk$ and $\nk$ aim to approximate $U^{(k)}$, $v^{(k)}$ and $b^{(k)}$, respectively, whereas $\Lk$ and $\wk$ correct the error on the boundary, which is introduced by $\Pk$. 
\end{definition}
As a result of Definition~\ref{def:variation_adjoint_state}, we get the following analogue to Lemma \ref{lem:recursion_Ueps}.
\begin{lemma}\label{lem:recursion_Peps_L2}
We have for $k\ge 2$
\begin{equation}
\PPi_\eps-\eps^{k-1}\Pk_\eps=\sum_{\ell=1}^{k-1}\eps^{\ell-1}\left(\eps^2\Pell+\eps^{d-2}\Lell_\eps\circ T_\eps+\eps^d\well\circ T_\eps+\eps^{d-2}\mell\circ T_\eps+\ln(\eps)\nell\circ T_\eps\right)\quad\text{ for }d=2,
\end{equation}
\begin{equation}
\PPi_\eps-\eps^{k-1}\Pk_\eps=\sum_{\ell=1}^{k-1}\eps^{\ell-1}\left(\eps^2\Pell+\eps^{d-2}\Lell_\eps\circ T_\eps+\eps^d\well\circ T_\eps+\eps^{d-2}\mell\circ T_\eps\right)\quad\text{ for }d=3,
\end{equation}
where we introduce for convenience $\PPi=\Li_\eps=\wi=\wii=\mi=\nni=0$.
\end{lemma}
\begin{proof}
This follows by a straight forward induction proof and is therefore omitted. 
\end{proof}
\begin{lemma}\label{lma:def_Pk}
For $k\ge 2$ let $\Pk(x):=-\alpha_1\int_\omega \phi(x-y)F^{(k)}(y)\; dy$, where
\begin{equation}
\phi(x):=\begin{cases}\frac{1}{8\pi}\left(|x|^2\ln(|x|)-|x|^2\right)\quad&\text{ for }d=2,\\
-\frac{1}{8\pi}|x|\quad&\text{ for }d=3,
\end{cases}
\end{equation}
 denotes the fundamental solution of the bi-harmonic equation (see \cite[Sec. 4.2, pp. 201]{b_GESH_1964a}) and $F^{(k)}$ is defined in Lemma \ref{lma:def_U}.
Then $\Pk$ satisfies
\begin{equation}\label{eq:approx_U}
\int_{\VR^d}\nabla\Pk\cdot\nabla \varphi\;dx=-\alpha_1\int_{\VR^d}U^{(k)}\varphi\; dx \quad\text{ for all }\varphi\in C^1_c(\VR^d).
\end{equation}
Moreover, in dimension $d=2$, $\Pk$ has the asymptotic behavior
\begin{equation}\label{eq:asymp_Pk_2}
    \begin{split}
        \Pk(x) = &  A^{(k)}_2(x)\ln(|x|)+A^{(k)}_1(x)\ln(|x|)+A^{(k)}_0(x)\ln(|x|) \\
                 & +B^{(k)}_2(x)+B^{(k)}_1(x)+B^{(k)}_0(x)+\sum_{\ell=1}^NS_\ell^{(k)}(x)+\mathcal{O}(|x|^{-(d-2+N+1)}),
\end{split}
\end{equation}
and in dimension $d=3$, we have
\begin{equation}\label{eq:asymp_Pk_3}
\Pk(x)=A^{(k)}_1(x)+A^{(k)}_0(x)+A^{(k)}_{-1}(x)+\sum_{\ell=1}^NS_\ell^{(k)}(x)+\mathcal{O}(|x|^{-(d-2+N+1)}),
\end{equation}
for $|x|\rightarrow \infty$, where $N\ge 1$, $A^{(k)}_i, B^{(k)}_i$ are homogeneous of degree $i$ and $S^{(k)}_i$ are homogeneous of degree \newline$-(d-2+i)$.
\end{lemma}
\begin{proof}
Similar to \cite[Lemma 4.1, pp. 54]{b_GITR_2001a}, one can show that $\Pk\in C^2(\VR^d)$ and \[\frac{\partial^2}{\partial x_i\partial x_j}\Pk=-\alpha_1\int_\omega \frac{\partial^2}{\partial x_i\partial x_j}\phi(x-y)F^{(k)}(y)\; dy.\]Since $-\Delta \phi=E$, where $E$ denotes the fundamental solution of the Laplace equation introduced in \eqref{eq:laplace_fundamental}, one readily follows
\[-\Delta \Pk=\alpha_1\int_\omega \Delta\phi(x-y)F^{(k)}(y)\; dy=-\alpha_1\int_\omega E(x-y)F^{(k)}(y)\; dy=-\alpha_1 U^{(k)}\quad\text{ on }\VR^d,\] where in the last step we used \eqref{eq:newton}. This shows \eqref{eq:approx_U}. Now, the asymptotic behaviour \eqref{eq:asymp_Pk_2},\eqref{eq:asymp_Pk_3} follows from a Taylor's expansion of $\phi$.
\end{proof}
\begin{example}
The terms $S^{(k)}_{\ell+1} :\VR^d\to \VR$ are given for $k\ge 2$ and $\ell \ge 1$ by 
\begin{equation}\label{eq:SK}
    S^{(k)}_{\ell}(x) = \frac{1}{(\ell+2)!} \int_\omega \partial^{\ell+2}_t \phi(x-ty)|_{t=0} F^{(k)}(y)\;dy,
\end{equation}
with $F^{(k)}(y):= \frac{1}{(k-2)!}\nabla^{k-2}\big(f_1-f_2\big)(x_0)[y]^{k-2}$. Additionally, the leading terms can be explicitly computed for $d=2$ as
\begin{equation}
A^{(k)}_2(x)=-\frac{\alpha_1}{8\pi}|x|^2 \int_\omega F^{(k)}(y)\;dy,
\end{equation}
\begin{equation}
A^{(k)}_1(x)=\frac{\alpha_1}{4\pi} \int_\omega x\cdot yF^{(k)}(y)\;dy,
\end{equation}
\begin{equation}
A^{(k)}_0(x)=-\frac{\alpha_1}{8\pi} \int_\omega|y|^2 F^{(k)}(y)\;dy,
\end{equation}
\begin{equation}
B^{(k)}_2(x)=\frac{\alpha_1}{8\pi}|x|^2 \int_\omega F^{(k)}(y)\;dy,
\end{equation}
\begin{equation}
B^{(k)}_1(x)=-\frac{\alpha_1}{8\pi} \int_\omega x\cdot y F^{(k)}(y)\;dy,
\end{equation}
\begin{equation}
B^{(k)}_0(x)=-\frac{\alpha_1}{8\pi} \int_\omega \left(\frac{(x\cdot y)^2}{|x|^2}-\frac{|y|^2}{2}\right)  F^{(k)}(y)\;dy,
\end{equation}
and similarly for $d=3$ the leading terms are given as
\begin{equation}
A^{(k)}_1(x)=\frac{\alpha_1}{8\pi}|x| \int_\omega F^{(k)}(y)\;dy,
\end{equation}
\begin{equation}
A^{(k)}_0(x)=-\frac{\alpha_1}{8\pi} \int_\omega \frac{x\cdot y}{|x|}F^{(k)}(y)\;dy,
\end{equation}
\begin{equation}
A^{(k)}_{-1}(x)=\frac{\alpha_1}{16\pi} \int_\omega \left(\frac{|x|^2|y|^2-(x\cdot y)^2}{|x|^3}\right)F^{(k)}(y)\;dy.
\end{equation}
\end{example}
Next we will look at the remaining components of \eqref{eq:variation_adjoint}.
\begin{definition}
For $k\ge3$ we define the corrector $\wk\in H^1(\Dsf)$ as the unique solution to
\begin{equation}\label{eq:wk_bd_L2}
\wk=-\sum_{\ell=1}^k S_\ell^{(k-\ell)}(x-x_0),\quad \text{ on }\Gamma,
\end{equation}
\begin{equation}\label{eq:wk_def}
\int_\Dsf\nabla \wk\cdot \nabla \varphi\, dx=\int_{\Sigma} \left(\sum_{\ell=1}^k\partial_\nu S_\ell^{(k-\ell)}(x-x_0)\right)\varphi\; dS\quad \text{ for all }\varphi\in H^1_\Gamma(\Dsf),
\end{equation}
where we introduce $S^{(0)}_\ell=S^{(1)}_\ell=0$ for $\ell\ge 1$. Additionally, we define for $k\ge 2$ the corrector $\mk\in H^1_\Gamma(\Dsf)$ as the unique solution to
\begin{equation}\label{eq:mk_def}
\int_\Dsf\nabla \mk\cdot \nabla \varphi\, dx=-\alpha_1\int_\Dsf v^{(k)}\varphi\; dx\quad \text{ for all }\varphi\in H^1_\Gamma(\Dsf),
\end{equation}
and $\nk\in H^1_\Gamma(\Dsf)$ as the unique solution to
\begin{equation}\label{eq:nk_def}
\int_\Dsf\nabla \nk\cdot \nabla \varphi\, dx=-\alpha_1b^{(k)}\int_\Dsf \varphi\; dx\quad \text{ for all }\varphi\in H^1_\Gamma(\Dsf).
\end{equation}
\end{definition}
The function $\Pk$ approximates $\Pk_\eps$ inside $\omega$ but introduces an error on the boundary $\partial \Dsf_\eps$ in a similar fashion to the approximation of the variation of the direct state $U^{(k)}_\eps$ by the boundary layer corrector $U^{(k)}$. The main difference now is that the asymptotic behaviour of $\Pk$ (see \eqref{eq:asymp_Pk_2}, \eqref{eq:asymp_Pk_3}) requires the first six terms in dimension $d=2$ and the first three terms for dimension $d=3$ to be corrected during each step. This can readily be done for homogeneous terms, whereas the logarithm occurring in dimension two causes some issues that require special attention. We will illustrate the procedure for the Dirichlet boundary $\Gamma_\eps$ (the Neumann boundary part is corrected in the same fashion) in the following. In view of the expansion \eqref{eq:asymp_Pk_2} of $\Pk$ in dimension $d=2$, we can correct the error produced by $\Pk$ on the boundary part $\Gamma_\eps$, by correcting the terms $\ln(|\cdot|)A^{(k)}_i(\cdot), B^{(k)}_i(\cdot)$, $i=0,1,2$ individually.

In order to correct a term $f(x)$ (e.g. $f(x) = B_1^{(k)}(x)$) on the boundary $\Gamma_\eps$, which we assume is 
homogeneous of degree $r$, we introduce a regular corrector $s\in H^1(\Dsf)$ defined on the fixed domain $\Dsf$, such that $s|_{\Gamma}=-f(x-x_0)$. Now rescaling $s$ to the domain $\Dsf_\eps$ yields $s\circ T_\eps|_{\Gamma_\eps}=-f(\eps x)$. Hence, since $f$ is homogeneous of degree $r$, we just need to scale $s$ by the factor $\eps^{-r}$ to get 
\[
    \eps^{-r}s\circ T_\eps|_{\Gamma_\eps}=-\eps^{-r}f(\eps x)=-f(x).
\]
In this way we can correct the terms $B_i^{(k)}$, $i=0,1,2$. 

Unfortunately, the terms $\ln(|x|)A_i^{(k)}(x)$, $i=0,1,2$ are not homogeneous, since the natural logarithm has the property 
\[
\ln(|\eps x|)=\ln(\eps)+\ln(|x|),\quad\text{ for all }\eps>0.
\]
Thus, correcting a term $f(x)\ln(|x|)$ (e.g. $f(x) = A_1^{(k)}(x)$), where $f$ is homogeneous of degree $r$, with a function $s$ defined on the fixed domain and boundary values $s|_\Gamma=-f(x-x_0)\ln(|x-x_0|)$ yields
\begin{equation}\label{eq:sk_scaling}
\eps^{-r}s\circ T_\eps|_{\Gamma_\eps}=-f(x)\ln(|x|)-f(x)\ln(\eps).
\end{equation}
Hence, the scaled function $s$ corrects $ f(x)\ln(|x|)$, but also introduces the new error $-f(x)\ln(\eps)$ on the boundary $\Gamma_\eps$. Fortunately, this term can be corrected by another function $s^\prime\in H^1(\Dsf)$ with boundary values $s^\prime|_\Gamma=f(x-x_0)$ scaled by the factor $\eps^{-r}\ln(\eps)$:
\[
    \eps^{-r}s\circ T_\eps|_{\Gamma_\eps}  +  \ln(\eps) \eps^{-r} s'\circ T_\eps|_{\Gamma_\eps} = - f(x) \ln(|x|).
\]
In this way we can correct every function $A^{(k)}_i, B^{(k)}_i$ appearing in \eqref{eq:asymp_Pk_2}, \eqref{eq:asymp_Pk_3}. 
Summarising, for $d=2$, we need 
to introduce one corrector equation for each function $B_0^{(k)},B_1^{(k)},B_2^{(k)}$ and 
two corrector equations for each function $\ln(|\cdot|) A^{(k)}_0(\cdot), \ln(|\cdot|) A^{(k)}_1(\cdot), \ln(|\cdot|) A^{(k)}_2(\cdot)$, which makes a total of $9$ corrector equations. This motivates the following definition.

\begin{definition}\label{def:sk}
For $k\ge 2$ and $d=2$ let
\begin{equation}\label{eq:Lk_def_2}
\Lk_\eps:=s^{(k)}_1+\ln(\eps)s^{(k)}_2+\eps s^{(k)}_3+\eps\ln(\eps)s^{(k)}_4+\eps^2s^{(k)}_5+\eps^2\ln(\eps)s^{(k)}_6+s^{(k)}_7+\eps s^{(k)}_8+\eps^2 s^{(k)}_9,
\end{equation}
where $s^{(k)}_i\in H^1(\Dsf)$, $i\in\{1,...,9\}$ are the unique solutions to the following set of equations:
\begin{enumerate}[label=(\roman*)]
    \item corrector equations for $A_2^{(k)}(x)\ln(|x|)$
    \begin{itemize}
\item[$\triangleright$] $s^{(k)}_1=-A^{(k)}_2(x-x_0)\ln(|x-x_0|)$ on $\Gamma$ and
\begin{equation}
\int_\Dsf\nabla s^{(k)}_1\cdot\nabla \varphi\; dx=\int_{\Sigma}\partial_\nu \left(A^{(k)}_2(x-x_0)\ln(|x-x_0|)\right)\varphi\; dS\quad\text{ for all }\varphi\in H^1_\Gamma(\Dsf).
\end{equation}
\item[$\triangleright$] $s^{(k)}_2=A^{(k)}_2(x-x_0)$ on $\Gamma$ and
\begin{equation}
\int_\Dsf\nabla s^{(k)}_2\cdot\nabla \varphi\; dx=-\int_{\Sigma}\partial_\nu \left(A^{(k)}_2(x-x_0)\right)\varphi\; dS\quad\text{ for all }\varphi\in H^1_\Gamma(\Dsf).
\end{equation}
\end{itemize}
\item corrector equations for $A_1^{(k)}(x)\ln(|x|)$
\begin{itemize}
\item[$\triangleright$] $s^{(k)}_3=-A^{(k)}_1(x-x_0)\ln(|x-x_0|)$ on $\Gamma$ and
\begin{equation}
\int_\Dsf\nabla s^{(k)}_3\cdot\nabla \varphi\; dx=\int_{\Sigma}\partial_\nu \left(A^{(k)}_1(x-x_0)\ln(|x-x_0|)\right)\varphi\; dS\quad\text{ for all }\varphi\in H^1_\Gamma(\Dsf).
\end{equation}
\item[$\triangleright$] $s^{(k)}_4=A^{(k)}_1(x-x_0)$ on $\Gamma$ and
\begin{equation}
\int_\Dsf\nabla s^{(k)}_4\cdot\nabla \varphi\; dx=-\int_{\Sigma}\partial_\nu \left(A^{(k)}_1(x-x_0)\right)\varphi\; dS\quad\text{ for all }\varphi\in H^1_\Gamma(\Dsf).
\end{equation}
\end{itemize}
\item corrector equations for $A^{(k)}_0(x)\ln(|x|)$
\begin{itemize}
\item[$\triangleright$] $s^{(k)}_5=-A^{(k)}_0(x-x_0)\ln(|x-x_0|)$ on $\Gamma$ and
\begin{equation}
\int_\Dsf\nabla s^{(k)}_5\cdot\nabla \varphi\; dx=\int_{\Sigma}\partial_\nu \left(A^{(k)}_0(x-x_0)\ln(|x-x_0|)\right)\varphi\; dS\quad\text{ for all }\varphi\in H^1_\Gamma(\Dsf).
\end{equation}
\item[$\triangleright$] $s^{(k)}_6=A^{(k)}_0(x-x_0)$ on $\Gamma$ and
\begin{equation}
\int_\Dsf\nabla s^{(k)}_6\cdot\nabla \varphi\; dx=-\int_{\Sigma}\partial_\nu \left(A^{(k)}_0(x-x_0)\right)\varphi\; dS\quad\text{ for all }\varphi\in H^1_\Gamma(\Dsf).
\end{equation}
\end{itemize}
\item corrector equation for $B^{(k)}_2(x)$
\begin{itemize}
\item[$\triangleright$] $s^{(k)}_7=-B^{(k)}_2(x-x_0)$ on $\Gamma$ and
\begin{equation}
\int_\Dsf\nabla s^{(k)}_7\cdot\nabla \varphi\; dx=\int_{\Sigma}\partial_\nu \left(B^{(k)}_2(x-x_0)\right)\varphi\; dS\quad\text{ for all }\varphi\in H^1_\Gamma(\Dsf).
\end{equation}
\end{itemize}
\item corrector equation for $B^{(k)}_1(x)$
\begin{itemize}
\item[$\triangleright$] $s^{(k)}_8=-B^{(k)}_1(x-x_0)$ on $\Gamma$ and
\begin{equation}
\int_\Dsf\nabla s^{(k)}_8\cdot\nabla \varphi\; dx=\int_{\Sigma}\partial_\nu \left(B^{(k)}_1(x-x_0)\right)\varphi\; dS\quad\text{ for all }\varphi\in H^1_\Gamma(\Dsf).
\end{equation}
\end{itemize}
\item corrector equation for $B^{(k)}_0(x)$
\begin{itemize}
\item[$\triangleright$] $s^{(k)}_9=-B^{(k)}_0(x-x_0)$ on $\Gamma$ and
\begin{equation}
\int_\Dsf\nabla s^{(k)}_9\cdot\nabla \varphi\; dx=\int_{\Sigma}\partial_\nu \left(B^{(k)}_0(x-x_0)\right)\varphi\; dS\quad\text{ for all }\varphi\in H^1_\Gamma(\Dsf).
\end{equation}
\end{itemize}
\end{enumerate}
\end{definition}
Thanks to the fact that the three leading terms of $P^{(k)}$, $k\ge2$ in dimension $d=3$ are homogeneous functions, we follow that each one can be treated by a single corrector equation. This makes a total of three corrector equations, which we will introduce in the following definition.
\begin{definition} \label{def_Leps3D}
For $k\ge 2$ and $d=3$ let
\begin{equation}\label{eq:Lk_def_3}
\Lk_\eps:=s^{(k)}_1+\eps s^{(k)}_2+\eps^2s^{(k)}_3,
\end{equation}
where $s^{(k)}_i\in H^1(\Dsf)$, $i\in\{1,2,3\}$ are the unique solutions to the following set of equations:
\begin{enumerate}[label=(\roman*)]
    \item corrector equation for $A^{(k)}_1(x)$
    \begin{itemize}
        \item[$\triangleright$] $s^{(k)}_1=-A^{(k)}_1(x-x_0)$ on $\Gamma$ and
\begin{equation}
\int_\Dsf\nabla s^{(k)}_1\cdot\nabla \varphi\; dx=\int_{\Sigma}\partial_\nu A^{(k)}_1(x-x_0)\varphi\; dS\quad\text{ for all }\varphi\in H^1_\Gamma(\Dsf).
\end{equation}
\end{itemize}
\item corrector equation for $A^{(k)}_0(x)$
\begin{itemize}
\item[$\triangleright$] $s^{(k)}_2=-A^{(k)}_0(x-x_0)$ on $\Gamma$ and
\begin{equation}
\int_\Dsf\nabla s^{(k)}_2\cdot\nabla \varphi\; dx=\int_{\Sigma}\partial_\nu A^{(k)}_0(x-x_0)\varphi\; dS\quad\text{ for all }\varphi\in H^1_\Gamma(\Dsf).
\end{equation}
\end{itemize}
\item corrector equation for $A^{(k)}_{-1}(x)$
\begin{itemize}
\item[$\triangleright$] $s^{(k)}_3=-A^{(k)}_{-1}(x-x_0)$ on $\Gamma$ and
\begin{equation}
\int_\Dsf\nabla s^{(k)}_3\cdot\nabla \varphi\; dx=\int_{\Sigma}\partial_\nu A^{(k)}_{-1}(x-x_0)\varphi\; dS\quad\text{ for all }\varphi\in H^1_\Gamma(\Dsf).
\end{equation}
\end{itemize}
\end{enumerate}
\end{definition}
\begin{lemma}\label{lem:Pk_recursion}
Let $k\ge 2$ and $\eps>0$ small. Then there holds
\begin{equation}
\eps^2\Pk+\eps^{d-2}\Lk_\eps\circ T_\eps=\eps^2\left(\sum_{\ell=1}^NS^{(k)}_\ell+\mathcal{O}(|x|^{-(d-2+N+1)})\right),\quad\text{ on }\Gamma_\eps,
\end{equation}
for $N\ge 1$.
\end{lemma}
\begin{proof}
  We only consider the case $d=2$ as the three dimensional case follows from similar arguments. We compute, using that $A^{(k)}_i,B^{(k)}_i$, $i=1,2,3$ are homogeneous and Definition~\ref{def:sk}, that on $\Gamma_\eps$ holds:
\begin{equation}
    \Lk_\eps \circ T_\eps = -\eps^2\left( A^{(k)}_2(x)\ln(|x|)+A^{(k)}_1(x)\ln(|x|)+A^{(k)}_0(x)\ln(|x|)
    +B^{(k)}_2(x)+B^{(k)}_1(x)+B^{(k)}_0(x)\right).
\end{equation}
Hence, the asymptotic behaviour \eqref{eq:asymp_Pk_2} of $\Pk$ yields the desired result in dimension $d=2$.
\end{proof}
As a result of this pointwise behaviour of $\Pk$, we get the following boundary estimates in the $L_2$ norm and $H^{\frac12}$ semi-norm:
\begin{corollary}\label{cor:Pk_boundary}
For $k\ge 2$ and $d=2$ we have
\begin{align}
\eps^{\frac{1}{2}}\|\Pk_\eps-\eps^2\Pk-\eps^{d-2}\Lk_\eps\circ T_\eps-\eps^d\wk\circ T_\eps-\eps^{d-2}\mk\circ T_\eps-\ln(\eps)\nk\circ T_\eps\|_{L_2(\Gamma_\eps)}\le& C \eps^{\frac{d}{2}+2},\\
|\Pk_\eps-\eps^2\Pk-\eps^{d-2}\Lk_\eps\circ T_\eps-\eps^d\wk\circ T_\eps-\eps^{d-2}\mk\circ T_\eps-\ln(\eps)\nk\circ T_\eps|_{H^{\frac{1}{2}}(\Gamma_\eps)}\le& C \eps^{\frac{d}{2}+2},\\
\|\eps^2\partial_\nu\Pk+\eps^{d-2}\partial_\nu\Lk_\eps\circ T_\eps\|_{L_2(\Sigma_\eps)}\le& C \eps^{\frac{d+1}{2}+2}.
\end{align}
Similarly, we have for $k\ge2$ and $d=3$
\begin{align}
\eps^{\frac{1}{2}}\|\Pk_\eps-\eps^2\Pk-\eps^{d-2}\Lk_\eps\circ T_\eps-\eps^d\wk\circ T_\eps-\eps^{d-2}\mk\circ T_\eps\|_{L_2(\Gamma_\eps)}\le& C \eps^{\frac{d}{2}+2},\\
|\Pk_\eps-\eps^2\Pk-\eps^{d-2}\Lk_\eps\circ T_\eps-\eps^d\wk\circ T_\eps-\eps^{d-2}\mk\circ T_\eps|_{H^{\frac{1}{2}}(\Gamma_\eps)}\le& C \eps^{\frac{d}{2}+2},\\
\|\eps^2\partial_\nu\Pk+\eps^{d-2}\partial_\nu\Lk_\eps\circ T_\eps\|_{L_2(\Sigma_\eps)}\le& C \eps^{\frac{d+1}{2}+2}.
\end{align}
\end{corollary}
\begin{proof}
We only show the estimate for dimension $d=2$ and refer to similar arguments for the three dimensional case. Note that we have $\mk=\nk=0$ on $\Gamma$.
Hence, using the recursion in Lemma~\ref{lem:Pk_recursion} and \eqref{eq:wk_bd_L2}, we can follow the steps of the proof of Corollary \ref{cor:Uk_boundary} to obtain the desired boundary estimates in dimension $d=2$. 
\end{proof}


We have now gathered all ingredients to prove the following main result establishing the remainder estimate for the asymptotic behaviour of $\Pk_\eps$.

\begin{theorem}\label{thm:asymptotic_Pk_L2}
Let $k\ge 1$, $\eps>0$ small and $\alpha\in (0,1)$. There is a constant $C>0$, such that
\begin{align}
    \|\Pk_\eps-\eps^2\Pk-\eps^{d-2}\Lk_\eps\circ T_\eps-\eps^d\wk\circ T_\eps-\eps^{d-2}\mk\circ T_\eps-\ln(\eps)\nk\circ T_\eps\|_\eps &\le C \eps^{1-\alpha} \quad \text{ for }d=2,\label{eq:main_adj_L2_2}\\
\|\Pk_\eps-\eps^2\Pk-\eps^{d-2}\Lk_\eps\circ T_\eps-\eps^d\wk\circ T_\eps-\eps^{d-2}\mk\circ T_\eps\|_\eps &\le C \eps \quad \text{ for }d=3.\label{eq:main_adj_L2_3}
\end{align}
\end{theorem}
\begin{proof}
We shall only give a sketch of the proof for $d=2$, as the idea is similar to the proof of Theorem \ref{thm:asymptotic_Uk}.
At first, we note that there holds
\begin{equation}
\int_{\Dsf_\eps}\nabla\PPi_\eps\cdot\nabla\varphi\; dx=-\eps^2\alpha_1\int_{\Dsf_\eps}\Ui_\eps\varphi\; dx\quad\text{ for all }\varphi\in H^1_{\Gamma_\eps}(\Dsf_\eps).
\end{equation}
Thus, an application of H\"older's inequality and the estimate $\|\Ui_\eps\|_\eps\le C\eps^{1-\alpha}$ yield
\begin{equation}
\|\PPi_\eps\|_\eps\le C\eps^{1-\alpha},
\end{equation}
for a positive constant $C>0$. Here, we additionally used  that $\PPi_\eps$ has homogeneous boundary values on $\Gamma_\eps$. Next, we seek a governing equation for $V^{(2)}_\eps:=\Pii_\eps-\eps^2 \Pii-\eps^{d-2}\Lii_\eps\circ T_\eps-\eps^{d-2}\mii\circ T_\eps-\ln(\eps)\nii \circ T_\eps$.
Rewriting the respective equations onto the scaled domain $\Dsf_\eps$ leaves us with
\begin{align}
\int_{\Dsf_\eps}\nabla V^{(2)}_\eps\cdot\nabla\varphi\; dx=&-\eps^2\alpha_1\int_{\Dsf_\eps}\left(\Uii_\eps-\Uii-\eps^{d-2}\vii\circ T_\eps-\ln(\eps)b^{(2)}\right)\varphi\; dx\\
&+\int_{\Sigma_\eps}\partial_\nu\left(\eps^2 \Pii+\eps^{d-2}\Lii\circ T_\eps\right)\varphi\; dS
\end{align}
for all $\varphi\in H^1_{\Gamma_\eps}(\Dsf_\eps)$.
Now we can deduce from Lemma \ref{lma:aux}, Corollary \ref{cor:Pk_boundary} and Theorem \ref{thm:asymptotic_Uk} that there is a positive constant $C>0$, such that
\begin{equation}
\|V^{(2)}_\eps\|_\eps\le C\eps^{1-\alpha},
\end{equation}
which shows the estimate \eqref{eq:main_adj_L2_2} for $k=2$. Now successively dividing by $\eps>0$ and subtracting the respective terms, one can readily check that the estimate holds for any $k\ge 3$.
\end{proof}
\section{Complete topological expansion - $L_2$ tracking-type} \label{sec:6}
In this section we compute the $n$-th topological derivative of the $L_2$ tracking-type part of the cost function defined in \eqref{eq:cost}. That is, we are deriving an asymptotic expansion of the form
\begin{equation}\label{eq:topo_def_L2}
\mathcal J_1(\Omega_\eps)=\mathcal J_1(\Omega)+\sum_{k=1}^n \ell_k(\eps) d^{k}\mathcal J_1(\Omega)(\omega, x_0)+o(\ell_n(\eps)),
\end{equation}
with $\mathcal J_1(\Omega)$ defined in \eqref{eq:cost_b}. Here $d^k\mathcal J_1(\Omega)(\omega, x_0)$ denotes the $k$-th topological derivative with respect to the initial domain $\Omega$ and perturbation shape $\omega$ at the point $x_0$ and $\ell_k:\VR^+\rightarrow \VR^+$ are continuous functions satisfying 
\[
\lim_{\eps\searrow 0}\ell_k(\eps)=0\quad \text{ and } \quad\lim_{\eps\searrow 0}\frac{\ell_{k+1}(\eps)}{\ell_k(\eps)}=0,\quad \text{ for }k\ge1.
\]
As we will see, the terms of logarithmic order, which occur in the asymptotic expansion of the adjoint state variable in $d=2$, lead to a differing topological derivative compared to dimension $d=3$. Thus, we will distinguish between both scenarios and derive a general formula of the topological derivative for both cases separately.
\subsection{General formula for higher order topological derivatives in $d=2$}
In this section we restrict ourselves to dimension $d=2$ and present the following result:
\begin{theorem}\label{thm:deriv_2_L2}
Let $\ell_1(\eps):=|\omega_\eps|$, $\ell_{2n}(\eps)=\eps^n \ln(\eps)|\omega_\eps|$ and $\ell_{2n+1}(\eps)=\eps^n |\omega_\eps|$, for $n\ge1$.
    The topological derivative of $\mathcal J_2$ at $x_0\in \Dsf\setminus\bar{\Omega}$ and $\omega\subset \VR^2$ with $0\in \omega$ in dimension $d=2$ is given by
\begin{equation}
d^1 \mathcal J_1(\Omega) (\omega, x_0)=\big((f_2-f_1)p_0\big)(x_0),\quad d^2 \mathcal J_1 (\Omega)(\omega, x_0)=0,
\end{equation}
    \begin{align}
        \begin{split}
            d^{2n+1}\mathcal J_1(\Omega) (\omega, x_0)  = &    \frac{1}{|\omega|}\frac{1}{n!}\int_\omega \nabla^n\big((f_2-f_1)p_0\big)(x_0)[x]^n \;dx\\
&+\frac{1}{|\omega|} \left(\sum_{j=0}^{n-4}\frac{1}{j!}\int_\omega \nabla^j\big(f_2-f_1\big)(x_0)[x]^j P^{(n-2-j)}(x)\;dx\right)\\
          & + \frac{1}{|\omega|}\left( \sum_{j=0}^{n-2}\frac{1}{j!}\int_\omega \nabla^j \big((f_2-f_1)s_1^{(n-j)}\big)(x_0)[x]^j\; dx\right)\\
          & + \frac{1}{|\omega|}\left( \sum_{j=0}^{n-3}\frac{1}{j!}\int_\omega \nabla^j \big((f_2-f_1)s_3^{(n-1-j)}\big)(x_0)[x]^j\; dx\right)\\
          & + \frac{1}{|\omega|}\left( \sum_{j=0}^{n-4}\frac{1}{j!}\int_\omega \nabla^j \big((f_2-f_1)s_5^{(n-2-j)}\big)(x_0)[x]^j\; dx\right)\\
          & + \frac{1}{|\omega|}\left( \sum_{j=0}^{n-2}\frac{1}{j!}\int_\omega \nabla^j \big((f_2-f_1)s_7^{(n-j)}\big)(x_0)[x]^j\; dx\right)\\
          & + \frac{1}{|\omega|}\left( \sum_{j=0}^{n-3}\frac{1}{j!}\int_\omega \nabla^j \big((f_2-f_1)s_8^{(n-1-j)}\big)(x_0)[x]^j\; dx\right)\\
          & + \frac{1}{|\omega|}\left( \sum_{j=0}^{n-4}\frac{1}{j!}\int_\omega \nabla^j \big((f_2-f_1)s_9^{(n-2-j)}\big)(x_0)[x]^j\; dx\right)\\
          & + \frac{1}{|\omega|}\left( \sum_{j=0}^{n-5}\frac{1}{j!}\int_\omega \nabla^j \big((f_2-f_1)w^{(n-2-j)}\big)(x_0)[x]^j\; dx\right)\\
          & + \frac{1}{|\omega|}\left( \sum_{j=0}^{n-2}\frac{1}{j!}\int_\omega \nabla^j \big((f_2-f_1)m^{(n-j)}\big)(x_0)[x]^j\; dx\right),
    \end{split}
   \end{align}
    \begin{align}
        \begin{split}
            d^{2n}\mathcal J_1(\Omega) (\omega, x_0)  = &    \frac{1}{|\omega|}\left( \sum_{j=0}^{n-2}\frac{1}{j!}\int_\omega \nabla^j \big((f_2-f_1)s_2^{(n-j)}\big)(x_0)[x]^j\; dx\right)\\
&+\frac{1}{|\omega|}\left( \sum_{j=0}^{n-3}\frac{1}{j!}\int_\omega \nabla^j \big((f_2-f_1)s_4^{(n-1-j)}\big)(x_0)[x]^j\; dx\right)\\
          & + \frac{1}{|\omega|}\left( \sum_{j=0}^{n-4}\frac{1}{j!}\int_\omega \nabla^j \big((f_2-f_1)s_6^{(n-2-j)}\big)(x_0)[x]^j\; dx\right)\\
&+\frac{1}{|\omega|}\left( \sum_{j=0}^{n-2}\frac{1}{j!}\int_\omega \nabla^j \big((f_2-f_1)n^{(n-j)}\big)(x_0)[x]^j\; dx\right),
    \end{split}
   \end{align}
for $n\ge1$, where for $\ell\ge1$, $i\in \{1,...,9\}$, $P^{(\ell)}$ are defined in Lemma~\ref{lma:def_Pk}, $s^{(\ell)}_i$ in Definition \ref{def:sk} and $w^{(\ell)},m^{(\ell)},n^{(\ell)}$ are defined in \eqref{eq:wk_def}, \eqref{eq:mk_def} and \eqref{eq:nk_def}, respectively.
\end{theorem}
\begin{proof}
    This can be shown similarly to Theorem \ref{thm:deriv_2_H1} by considering the difference of the Lagrangian $\mathcal L$ defined in \eqref{eq:lagrangian_L2} using the 
averaged adjoint variable $p_\eps$ and using the error estimate proved in Theorem~\ref{thm:asymptotic_Pk_L2}. A detailed proof can be found in the Appendix \ref{sec:7}.
\end{proof}
Again, we finish the section by stating the first five topological derivatives. 
\begin{corollary}\label{cor:formula_2_L2}
    The first five terms of the topological expansion in dimension $d=2$ read:
\begin{align*}
    d^1\mathcal J_{1}(\Omega)(\omega, x_0)=&\big((f_2-f_1)p_0\big)(x_0), && \ell_1(\eps) = |\omega| \eps^2, \\
    d^2\mathcal J_{1}(\Omega)(\omega, x_0)=&0,  &&\ell_2(\eps)  = |\omega|  \eps^2(\eps\ln(\eps)), \\
    d^3\mathcal J_{1}(\Omega)(\omega, x_0)=&\frac{1}{|\omega|}\int_\omega\nabla\big((f_2-f_1)p_0\big)(x_0)[x]\; dx, &&  \ell_3(\eps) = |\omega| \eps^3,\\
d^4\mathcal J_{1}(\Omega)(\omega, x_0)=&\big((f_1-f_2)s_2^{(2)}\big)(x_0)+\big((f_1-f_2)n^{(2)}\big)(x_0), && \ell_4(\eps) = |\omega|\eps^3 (\eps\ln(\eps)),\\
d^5\mathcal J_{1}(\Omega)(\omega, x_0)=&\frac{1}{2|\omega|}\int_\omega\nabla^2\big((f_2-f_1)p_0\big)(x_0)[x]^2\; dx && \ell_5(\eps) = |\omega| \eps^4,\\
&+\big((f_1-f_2)s_1^{(2)}\big)(x_0)+\big((f_1-f_2)s_7^{(2)}\big)(x_0) \\
& +\big((f_1-f_2)m^{(2)}\big)(x_0).
\end{align*}
\end{corollary}
\subsection{General formula for higher order topological derivatives in $d=3$}
Here, we will give an analogous result to Theorem \ref{thm:deriv_2_L2} in three space dimensions.
\begin{theorem}\label{thm:deriv_3_L2}
For $n\ge 1$ let $\ell_n(\eps):=\eps^{n-1}|\omega_\eps|$.
    The topological derivative of $\mathcal J_1$ at $x_0\in \Dsf\setminus\bar{\Omega}$ and $\omega\subset \VR^3$ with $0\in \omega$ in dimension $d=3$ is given by
    \begin{align}
        \begin{split}
            d^{n+1}\mathcal J_1(\Omega) (\omega, x_0)  = &    \frac{1}{|\omega|}\frac{1}{n!}\int_\omega \nabla^n\big((f_2-f_1)p_0\big)(x_0)[x]^n \;dx\\
&+\frac{1}{|\omega|} \left(\sum_{j=0}^{n-4}\frac{1}{j!}\int_\omega \nabla^j\big(f_2-f_1\big)(x_0)[x]^j P^{(n-2-j)}(x)\;dx\right)\\
          & + \frac{1}{|\omega|}\left( \sum_{j=0}^{n-3}\frac{1}{j!}\int_\omega \nabla^j \big((f_2-f_1)s_1^{(n-1-j)}\big)(x_0)[x]^j\; dx\right)\\
          & + \frac{1}{|\omega|}\left( \sum_{j=0}^{n-4}\frac{1}{j!}\int_\omega \nabla^j \big((f_2-f_1)s_2^{(n-2-j)}\big)(x_0)[x]^j\; dx\right)\\
          & + \frac{1}{|\omega|}\left( \sum_{j=0}^{n-5}\frac{1}{j!}\int_\omega \nabla^j \big((f_2-f_1)s_3^{(n-3-j)}\big)(x_0)[x]^j\; dx\right)\\
          & + \frac{1}{|\omega|}\left( \sum_{j=0}^{n-6}\frac{1}{j!}\int_\omega \nabla^j \big((f_2-f_1)w^{(n-3-j)}\big)(x_0)[x]^j\; dx\right)\\
          & + \frac{1}{|\omega|}\left( \sum_{j=0}^{n-3}\frac{1}{j!}\int_\omega \nabla^j \big((f_2-f_1)m^{(n-1-j)}\big)(x_0)[x]^j\; dx\right),
    \end{split}
   \end{align}
for $n\ge0$, where for $\ell\ge1$, $i\in \{1,...,3\}$, $P^{(\ell)}$ are defined in Lemma \ref{lma:def_Pk}, $s^{(\ell)}_i$ in Definition \ref{def_Leps3D} and $w^{(\ell)},m^{(\ell)}$ are defined in \eqref{eq:wk_def} and \eqref{eq:mk_def}, respectively.
\end{theorem}
Similarly to Corollary \ref{cor:formula_2_L2} we deduce the following result.
\begin{corollary}\label{cor:formula_3_L2}
The first five terms of the topological derivative of the $L_2$ cost functional for dimension $d=3$ are given as
\begin{align*}
    d^1\mathcal J_{1}(\Omega)(\omega, x_0)=&\big((f_2-f_1)p_0\big)(x_0), && \ell_1(\eps) = |\omega|\eps^3, \\
d^2\mathcal J_{1}(\Omega)(\omega, x_0)=&\frac{1}{|\omega|}\int_\omega\nabla\big((f_2-f_1)p_0\big)(x_0)[x]\; dx, && \ell_2(\eps) = |\omega|\eps^4,\\
d^3\mathcal J_{1}(\Omega)(\omega, x_0)=&\frac{1}{2|\omega|}\int_\omega\nabla^2\big((f_2-f_1)p_0\big)(x_0)[x]^2\; dx, && \ell_3(\eps) = |\omega|\eps^5,\\
d^4\mathcal J_{1}(\Omega)(\omega, x_0)=&\frac{1}{6|\omega|}\int_\omega\nabla^3\big((f_2-f_1)p_0\big)(x_0)[x]^3\; dx, && \ell_4(\eps) = |\omega|\eps^6,\\
&+\big((f_1-f_2)s_1^{(2)}\big)(x_0)+\big((f_1-f_2)m^{(2)}\big)(x_0),\\
d^5\mathcal J_{1}(\Omega)(\omega, x_0)=&\frac{1}{24|\omega|}\int_\omega\nabla^4\big((f_2-f_1)p_0\big)(x_0)[x]^4\; dx, && \ell_5(\eps) = |\omega|\eps^7,\\
&+\big((f_1-f_2)s_2^{(2)}\big)(x_0)+\frac{\big(f_2-f_1\big)(x_0)}{|\omega|}\int_\omega\Pii\; dx && \\
&+\big((f_1-f_2)s_1^{(3)}\big)(x_0)+\big((f_1-f_2)m^{(3)}\big)(x_0)\\
&+\frac{1}{|\omega|}\int_\omega\nabla\big((f_2-f_1)s_1^{(2)}\big)(x_0)[x]\; dx\\
& +\frac{1}{|\omega|}\int_\omega\nabla\big((f_2-f_1)m^{(2)}\big)(x_0)[x]\; dx.
\end{align*}
\end{corollary}
\section{Conclusion}
%
In this paper we derived general formulas for the topological derivative of arbitrary order in two and three spatial dimensions for a simple model problem where the right hand side of the state equation is perturbed. This was done by employing a Lagrangian framework based on an averaged adjoint equation which allows for a systematic and iterative derivation of arbitrary order topological derivatives. The main ingredients are the topological asymptotic expansions of the state and averaged adjoint variable. We saw that these asymptotics differ in two and three space dimensions due to a different asymptotic behaviour of the fundamental solution of the Laplace equation, requiring additional corrector terms in the case $d=2$. While the asymptotic expansion of the averaged adjoint variable basically coincides with that of the state variable in the case of an $H^1$ tracking-type cost function, the case of an $L^2$ cost function is more complex and involves the fundamental solution of the bi-harmonic equation. We expect that similar techniques can also be useful for topological asymptotic expansions in the presence of PDE constraints involving a zero-order reaction term.

While the present work only considered a simple model problem where only the right hand side of the PDE is perturbed, we believe that the presented procedure can be a good starting point for more practically interesting problems, e.g. involving a perturbation of the principal part of the PDE operator.
In future work, we plan to investigate the use of higher order topological derivatives for further improving reconstruction algorithms, e.g. in applications of medical imaging or inverse gravimetry, or for accelerating level-set based topology optimisation algorithms.
\subsection*{Acknowledgements}
Phillip Baumann has been funded by the Austrian Science Fund (FWF) project P32911
\section*{Appendix}\label{sec:7}
In this section we derive the formulas for the topological derivative of the $L_2$ tracking-type cost functional in dimension $d=2$.
\begin{proof}
Recall the Lagrangian introduced in Section \ref{sec:3.1b}. Let $\eps \ge 0$. We first 
observe by testing \eqref{eq:averaged_adjoint_abstract_b} with $\varphi =u_\eps -u_0$ that
\begin{equation}
    \mathcal J_{1}(\Omega_\eps)  = \FL(\eps,u_\eps,p_\eps) = \FL(\eps,u_0,p_\eps)
\end{equation}
so that the cost function can be written only in terms of the averaged adjoint variable. Therefore, we have
\begin{equation}
    \mathcal J_{1}(\Omega_\eps) - \mathcal J_{1}(\Omega)  = \FL(\eps,u_0,p_\eps) - \FL(\eps,u_0,p_0) + \FL(\eps,u_0,p_0) - \FL(0,u_0,p_0).
\end{equation}
Using Theorem~\ref{thm:asymptotic_Pk_L2}, we now derive an expansion for both differences on the right hand side.
\paragraph{Expansion of $\FL(\eps,u_0,p_0) - \FL(0,u_0,p_0)$:} 
The expansion of this difference is obtained by a Taylor expansion as shown in \eqref{eq:diff_Lagrangian_averaged1_H1}-\eqref{eq:deriv_part1_H1}: 
\begin{align}
    \frac{\FL(\eps,u_0,p_0) - \FL(0,u_0,p_0)}{|\omega_\eps|} &  =  \hat{p}_0(x_0) +  \sum_{n=1}^N \eps^n \frac{1}{|\omega|}\frac{1}{n!}\int_\omega \nabla^n\hat{p}_0(x_0)[x]^n \;dx + \Co(\eps^{N+1}), \quad N\ge 1.
\end{align}

\paragraph{Expansion of $\FL(\eps,u_0,p_\eps) - \FL(\eps,u_0,p_0)$:}

We proceed as in \eqref{eq:diff_Lagrangian_averaged} and obtain with the definition $\eps\PPi_\eps(x) = \left[p_\eps - p_0\right]\circ T_\eps$ and a change of variables:
\begin{align}\label{eq:diff_Lagrangian_averaged2}
    \FL(\eps,u_0,p_\eps) - \FL(\eps,u_0,p_0)  =\eps^d \int_\omega \left(f_2 - f_1\right)\circ T_\eps \eps\PPi_\eps\;dx.
\end{align}
Substituting $\eps\PPi_\eps$ by the recursion formula of Lemma~\ref{lem:recursion_Peps_L2} leads to
\begin{align}
    \int_\omega \left (f_2 - f_1\right) \circ T_\eps \eps\PPi_\eps\;dx =&  \sum_{n=1}^{N} \int_\omega \left(f_2 - f_1\right) \circ T_\eps\eps^{n+2} P^{(n)}\;dx \label{eq:first_term_L2}\\
										& + \sum_{n=1}^{N}\int_\omega \left(f_2 - f_1\right) \circ T_\eps \eps^n L_\eps^{(n)}\circ T_\eps \;dx\label{eq:second_term_L2} \\
                                                               & + \sum_{n=1}^{N}\int_\omega \left(f_2 - f_1\right) \circ T_\eps \eps^{n+2} w^{(n)}\circ T_\eps \;dx\label{eq:third_term_L2} \\
										& + \sum_{n=1}^{N}\int_\omega \left(f_2 - f_1\right) \circ T_\eps \eps^n m^{(n)}\circ T_\eps \;dx\label{eq:fourth_term_L2} \\
                                                              & + \sum_{n=1}^{N}\int_\omega \left(f_2 - f_1\right) \circ T_\eps \eps^n \ln(\eps)n^{(n)}\circ T_\eps \;dx\label{eq:fifth_term_L2} \\
                                                               & + \int_\omega \left(f_2 - f_1\right) \circ T_\eps \eps^{N+1}P_\eps^{(N+1)} \;dx. \label{eq:sixth_term_L2}
\end{align}

Now we can expand all six terms:    
\begin{itemize}
    \item First term \eqref{eq:first_term_L2}: We use Taylor's expansion to write:
        \begin{align}
            \left(f_2-f_1\right)\circ T_\eps  = \sum_{j=0}^{N} \eps^j a_j(x) + \Co(\eps^{N+1};x), \qquad a_j(x) := \frac{\nabla^j\big(f_2-f_1\big)(x_0)[x]^j}{j!}
        \end{align}
        We set for the proof $P^{(0)}:=0$, and $P^{(j)}:= 0$ and $a_j:=0$ for all $j > N$. Then, by Lemma \ref{lma:cauchy_landau} we have
        \begin{equation}
            \begin{split}
                \left( f_2-f_1\right)\circ T_\eps\left( \sum_{n=1}^{N} \eps^{n+2} P^{(n)} \right)  = \sum_{n=0}^{N-2}  \eps^{n+2} \left(\sum_{j=0}^n a_j(x) P^{(n-j)}(x)\right) + \Co(\eps^{N+1};x),
        \end{split}
        \end{equation}
        and therefore
        \begin{equation}\label{eq:remainder_dJk_1_L2}
            \sum_{n=1}^{N} \int_\omega \left(f_2-f_1\right)\eps^{n+2} P^{(n)}\;dx = \sum_{n=4}^{N}  \eps^{n} \int_\omega \left(\sum_{j=0}^{n-4} a_j(x) P^{(n-2-j)}(x)\right)\;dx + \Co(\eps^{N+1}),
        \end{equation}
where we took into account $p^{(0)}=p^{(1)}=0$.
\item Second term \eqref{eq:second_term_L2}:
In order to derive the correct formula, we first need to split the corrector $\Lell_\eps$ into its components. That is, we have
\begin{equation}
\Lk_\eps:=s^{(\ell)}_1+\ln(\eps)s^{(\ell)}_2+\eps s^{(\ell)}_3+\eps\ln(\eps)s^{(\ell)}_4+\eps^2s^{(\ell)}_5+\eps^2\ln(\eps)s^{(\ell)}_6+s^{(\ell)}_7+\eps s^{(\ell)}_8+\eps^2 s^{(\ell)}_9.
\end{equation}
Next, we use Taylor's formula to expand the functions $\eps\mapsto \hat{s}_\ell^{(n)}\circ T_\eps$  with $\hat{s}_\ell^{(n)}:=\left((f_2-f_1)s_\ell^{(n)}\right)$, $n\ge1$ and $\ell\in\{1,...,9\}$ to deduce
    \begin{equation}\label{eq:taylor_hat_s1}
        \hat{s}_\ell^{(n)} \circ T_\eps(x) = \sum_{j=0}^{N} \eps^j b_{\ell,j}^{(n)}(x) + \mathcal O(\eps^{N+1};x) \qquad b_{\ell,j}^{(n)}(x) := \frac{\nabla^j \hat s_\ell^{(n)}(x_0)[x]^j}{j!}.
    \end{equation}
Hence, a similar computation to the previous one yields
    \begin{equation}\label{eq:remainder_dJk_2_L2_1}
    \begin{split}
        \sum_{n=1}^{N}\int_\omega \eps^n \left((f_2-f_1)s_1^{(n)}\right)\circ T_\eps(x) \;dx =& \sum_{n=1}^{N} \eps^n \left( \sum_{j=0}^{N} \eps^j \int_\omega b_{1,j}^{(n)}(x)\;dx \right) + \Co(\eps^{N+1})\\
=&\sum_{n=2}^N \eps^n \left( \sum_{j=0}^{n-2}\int_\omega b_{1,j}^{(n-j)}(x)\; dx\right)+\Co(\eps^{N+1}),
    \end{split}
    \end{equation}
where we took into account that $s_1^{(1)}=0$ and therefore $b^{(1)}_{1,j}=0$ for $j\ge0$ as well. With the same arguments we get
    \begin{equation}\label{eq:remainder_dJk_2_L2_2}
    \begin{split}
        \sum_{n=1}^{N}\int_\omega \ln(\eps)\eps^n \left((f_2-f_1)s_2^{(n)}\right)\circ T_\eps \;dx
=&\sum_{n=2}^N \ln(\eps)\eps^n \left( \sum_{j=0}^{n-2}\int_\omega b_{2,j}^{(n-j)}(x)\; dx\right)+o(\eps^{N}),
    \end{split}
    \end{equation}
    \begin{equation}\label{eq:remainder_dJk_2_L2_3}
    \begin{split}
        \sum_{n=1}^{N}\int_\omega \eps^{n+1} \left((f_2-f_1)s_3^{(n)}\right)\circ T_\eps \;dx
=&\sum_{n=3}^{N} \eps^{n} \left( \sum_{j=0}^{n-3}\int_\omega b_{3,j}^{(n-1-j)}(x)\; dx\right)+\Co(\eps^{N+1}),
    \end{split}
    \end{equation}
    \begin{equation}\label{eq:remainder_dJk_2_L2_4}
    \begin{split}
        \sum_{n=1}^{N}\int_\omega \ln(\eps)\eps^{n+1} \left((f_2-f_1)s_4^{(n)}\right)\circ T_\eps \;dx
=&\sum_{n=3}^{N} \ln(\eps)\eps^{n} \left( \sum_{j=0}^{n-3}\int_\omega b_{4,j}^{(n-1-j)}(x)\; dx\right)+o(\eps^{N}),
    \end{split}
    \end{equation}
    \begin{equation}\label{eq:remainder_dJk_2_L2_5}
    \begin{split}
        \sum_{n=1}^{N}\int_\omega \eps^{n+2} \left((f_2-f_1)s_5^{(n)}\right)\circ T_\eps \;dx
=&\sum_{n=4}^{N} \eps^{n} \left( \sum_{j=0}^{n-4}\int_\omega b_{5,j}^{(n-2-j)}(x)\; dx\right)+\Co(\eps^{N+1}),
    \end{split}
    \end{equation}
    \begin{equation}\label{eq:remainder_dJk_2_L2_6}
    \begin{split}
        \sum_{n=1}^{N}\int_\omega \ln(\eps)\eps^{n+2} \left((f_2-f_1)s_6^{(n)}\right)\circ T_\eps \;dx
=&\sum_{n=4}^{N} \ln(\eps)\eps^{n} \left( \sum_{j=0}^{n-4}\int_\omega b_{6,j}^{(n-2-j)}(x)\; dx\right)+o(\eps^{N}),
    \end{split}
    \end{equation}
    \begin{equation}\label{eq:remainder_dJk_2_L2_7}
    \begin{split}
        \sum_{n=1}^{N}\int_\omega \eps^{n} \left((f_2-f_1)s_7^{(n)}\right)\circ T_\eps \;dx
=&\sum_{n=2}^{N} \eps^{n} \left( \sum_{j=0}^{n-2}\int_\omega b_{7,j}^{(n-j)}(x)\; dx\right)+\Co(\eps^{N+1}),
    \end{split}
    \end{equation}
    \begin{equation}\label{eq:remainder_dJk_2_L2_8}
    \begin{split}
        \sum_{n=1}^{N}\int_\omega \eps^{n+1} \left((f_2-f_1)s_8^{(n)}\right)\circ T_\eps \;dx
=&\sum_{n=3}^{N} \eps^{n} \left( \sum_{j=0}^{n-3}\int_\omega b_{8,j}^{(n-1-j)}(x)\; dx\right)+\Co(\eps^{N+1}),
    \end{split}
    \end{equation}
    \begin{equation}\label{eq:remainder_dJk_2_L2_9}
    \begin{split}
        \sum_{n=1}^{N}\int_\omega \eps^{n+2} \left((f_2-f_1)s_9^{(n)}\right)\circ T_\eps \;dx
=&\sum_{n=4}^{N} \eps^{n} \left( \sum_{j=0}^{n-4}\int_\omega b_{9,j}^{(n-2-j)}(x)\; dx\right)+\Co(\eps^{N+1}).
    \end{split}
    \end{equation}
\item Third term \eqref{eq:third_term_L2}: Using the Taylor expansion of $\eps\mapsto \hat{w}^{(n)}\circ T_\eps$ at $\eps=0$ with $\hat{w}^{(n)}:=(f_2-f_1)w^{(n)}$:
    \begin{equation}\label{eq:taylor_hat_w}
        \hat{w}^{(n)} \circ T_\eps(x) = \sum_{j=0}^{N} \eps^j \hat w_{j}^{(n)}(x) + \mathcal O(\eps^{N+1};x), \qquad \hat w_{j}^{(n)}(x) := \frac{\nabla^j \hat w^{(n)}(x_0)[x]^j}{j!}
    \end{equation}
for $n\ge 1$ one computes analogously
    \begin{equation}\label{eq:remainder_dJk_3_L2}
    \begin{split}
        \sum_{n=1}^{N}\int_\omega \eps^{n+2} \left((f_2-f_1)w^{(n)}\right)\circ T_\eps \;dx
=&\sum_{n=5}^{N} \eps^{n} \left( \sum_{j=0}^{n-5}\int_\omega \hat w_{j}^{(n-2-j)}(x)\; dx\right)+\Co(\eps^{N+1}),
    \end{split}
    \end{equation}
where we took into account that $w^{(1)}=w^{(2)}=0$, which explains the index shift. 
\item Fourth term \eqref{eq:fourth_term_L2}: Similarly, the expansion of $\eps \mapsto \hat{m}^{(n)}\circ T_\eps$ at $\eps=0$ with $\hat{m}^{(n)}:=(f_2-f_1)m^{(n)}$ reads
    \begin{equation}\label{eq:taylor_hat_m}
        \hat{m}^{(n)} \circ T_\eps(x) = \sum_{j=0}^{N} \eps^j \hat m_{j}^{(n)}(x) + \mathcal O(\eps^{N+1};x),\qquad \hat m_{j}^{(n)}(x) := \frac{\nabla^j \hat m^{(n)}(x_0)[x]^j}{j!}
    \end{equation}
for $n\ge 1$ and yields
    \begin{equation}\label{eq:remainder_dJk_4_L2}
    \begin{split}
        \sum_{n=1}^{N}\int_\omega \eps^{n} \left((f_2-f_1)m^{(n)}\right)\circ T_\eps \;dx
=&\sum_{n=2}^{N} \eps^{n} \left( \sum_{j=0}^{n-2}\int_\omega \hat m_{j}^{(n-j)}(x)\; dx\right)+\Co(\eps^{N+1}).
    \end{split}
    \end{equation}
\item Fifth term \eqref{eq:fifth_term_L2}: Again, we have by a Taylor's expansion of $\eps \mapsto \hat{n}^{(n)}\circ T_\eps$ at $\eps=0$ with $\hat{n}^{(n)}:=(f_2-f_1)n^{(n)}$ that:
    \begin{equation}\label{eq:taylor_hat_n}
        \hat{n}^{(n)} \circ T_\eps(x) = \sum_{j=0}^{N} \eps^j \hat n_{j}^{(n)}(x) + \mathcal O(\eps^{N+1};x), \qquad \hat n_{j}^{(n)}(x) := \frac{\nabla^j \hat n^{(n)}(x_0)[x]^j}{j!}
    \end{equation}
for $n\ge 1$ and therefore it follows that:
    \begin{equation}\label{eq:remainder_dJk_5_L2}
    \begin{split}
        \sum_{n=1}^{N}\int_\omega \ln(\eps)\eps^{n} \left((f_2-f_1)n^{(n)}\right)\circ T_\eps \;dx
=&\sum_{n=2}^{N} \ln(\eps)\eps^{n} \left( \sum_{j=0}^{n-2}\int_\omega \hat n_{j}^{(n-j)}(x)\; dx\right)+o(\eps^{N}).
    \end{split}
    \end{equation}
\item Sixth term \eqref{eq:sixth_term_L2}: Applying Lemma~\ref{lem:scaling_inequalities}, item (c), to the last term and using the asymptotics derived in Theorem~\ref{thm:asymptotic_Pk_L2} gives
\begin{align}\label{eq:remainder_dJk_6_L2}
\begin{split}
    \left|\int_\omega \left(f_1 - f_2\right)\circ T_\eps \eps^{N+1} P_\eps^{(N+1)}\;dx\right| \le& C\eps^{N-\alpha}\|\eps P_\eps^{(N+1)}\|_\eps \le C \eps^{N+1-2\alpha},
\end{split}
\end{align}
for a constant $C>0$ and $\alpha\in (0,1)$ sufficiently small.
\end{itemize}
Now combining \eqref{eq:remainder_dJk_1_L2} - \eqref{eq:remainder_dJk_6_L2} shows the formula given in Theorem \ref{thm:deriv_2_L2}.

\end{proof}

\bibliography{topo_refs}

\newcommand{\etalchar}[1]{$^{#1}$}
\providecommand{\bysame}{\leavevmode\hbox to3em{\hrulefill}\thinspace}
\providecommand{\MR}{\relax\ifhmode\unskip\space\fi MR }
\providecommand{\MRhref}[2]{%
  \href{http://www.ams.org/mathscinet-getitem?mr=#1}{#2}
}
\providecommand{\href}[2]{#2}
\begin{thebibliography}{MNP12b}

\bibitem[AA06]{a_AMAN_2006a}
S.~Amstutz and H.~Andrä, \emph{A new algorithm for topology optimization using
  a level-set method}, J. Comput. Phys. \textbf{216} (2006), no.~2, 573--588.
  \MR{2235384}

\bibitem[AB17]{a_AMBO_2017a}
S.~Amstutz and A.~Bonnaf{\'{e}}, \emph{Topological derivatives for a class of
  quasilinear elliptic equations}, Journal de Math{\'{e}}matiques Pures et
  Appliqu{\'{e}}es \textbf{107} (2017), no.~4, 367--408.

\bibitem[AG19]{a_AMGA_2019a}
S.~Amstutz and P.~Gangl, \emph{Topological derivative for the nonlinear
  magnetostatic problem}, Electron. Trans. Numer. Anal. \textbf{51} (2019),
  169--218. \MR{3987793}

\bibitem[Ams03]{phd_AM_2003a}
S.~Amstutz, \emph{Aspects th\'eoriques et num\'eriques en optimisation de forme
  topologique}, Ph.D. thesis, L'institut National des Sciences Appliqu\'ees de
  Toulouse, 2003.

\bibitem[Ams06]{a_AM_2006b}
\bysame, \emph{Topological sensitivity analysis for some nonlinear {PDE}
  systems}, Journal de Math{\'{e}}matiques Pures et Appliqu{\'{e}}es
  \textbf{85} (2006), no.~4, 540--557.

\bibitem[AN10]{a_AMNO_2010a}
S.~Amstutz and A.~A. Novotny, \emph{Topological asymptotic analysis of the
  kirchhoff plate bending problem}, ESAIM: Control, Optimisation and Calculus
  of Variations \textbf{17} (2010), no.~3, 705--721.

\bibitem[BC17]{a_BOCO_2017a}
M.~Bonnet and R.~Cornaggia, \emph{Higher order topological derivatives for
  three-dimensional anisotropic elasticity}, ESAIM: Mathematical Modelling and
  Numerical Analysis \textbf{51} (2017), no.~6, 2069--2092.

\bibitem[BMR17]{a_BEMARA_2017a}
E.~Beretta, A.~Manzoni, and L.~Ratti, \emph{A reconstruction algorithm based on
  topological gradient for an inverse problem related to a semilinear elliptic
  boundary value problem}, Inverse Problems \textbf{33} (2017), no.~3, 035010.

\bibitem[BS21]{a_BAST_2021a}
P.~Baumann and K.~Sturm, \emph{Computation of second order topological
  derivatives with application to linear elasticity}, preprint soon, 2021.

\bibitem[Del18]{c_DE_2018b}
M.~C. Delfour, \emph{Control, shape, and topological derivatives via minimax
  differentiability of {L}agrangians}, Springer INdAM Series, Springer
  International Publishing, 2018, pp.~137--164.

\bibitem[EKS94]{a_ESKOSC_1994a}
H.~A. Eschenauer, V.~V. Kobelev, and A.~Schumacher, \emph{Bubble method for
  topology and shape optimization of structures}, Structural Optimization
  \textbf{8} (1994), no.~1, 42--51.

\bibitem[GG01]{b_GITR_2001a}
D.~Gilbarg and N.~Grudinger, \emph{Elliptic partial differential equations of
  second order}, Springer, Berlin New York, 2001.

\bibitem[GGM01]{a_GAGUMA_2001a}
S.~Garreau, P.~Guillaume, and M.~Masmoudi, \emph{The topological asymptotic for
  {PDE} systems: The elasticity case}, {SIAM} Journal on Control and
  Optimization \textbf{39} (2001), 1756--1778.

\bibitem[GS64]{b_GESH_1964a}
IM~Gelfand and GE~Shilov, \emph{Generalized functions, vol. 1 academic press},
  vol. 1967, 1964.

\bibitem[GS20]{a_GAST_2020a}
P.~Gangl and K.~Sturm, \emph{A simplified derivation technique of topological
  derivatives for quasi-linear transmission problems}, ESAIM Control Optim.
  Calc. Var. \textbf{26} (2020), Paper No. 106, 20. \MR{4185062}

\bibitem[GS21]{a_GAST_2021a}
\bysame, \emph{Asymptotic analysis and topological derivative for 3d
  quasi-linear magnetostatics}, ESAIM Math. Model. Numer. Anal. \textbf{55}
  (2021), no.~suppl., S853--S875. \MR{4221309}

\bibitem[HL08]{a_HILA_2008a}
M.~Hinterm\"uller and A.~Laurain, \emph{Electrical impedance tomography: from
  topology to shape}, Control and Cybernetics \textbf{37} (2008), no.~4,
  913--933 (eng).

\bibitem[HL11]{a_HILA_2011a}
M.~Hintermüller and A.~Laurain, \emph{Optimal shape design subject to elliptic
  variational inequalities}, SIAM Journal on Control and Optimization
  \textbf{49} (2011), no.~3, 1015--1047.

\bibitem[HLN11]{a_HILANO_2011a}
M.~Hinterm\"uller, A.~Laurain, and A.~A. Novotny, \emph{Second-order
  topological expansion for electrical impedance tomography}, Advances in
  Computational Mathematics \textbf{36} (2011), no.~2, 235--265.

\bibitem[HM04]{a_HAMA_2004a}
M.~Hassine and M.~Masmoudi, \emph{The topological asymptotic expansion for the
  quasi-{S}tokes problem}, ESAIM Control Optim. Calc. Var. \textbf{10} (2004),
  no.~4, 478--504. \MR{2111076}

\bibitem[INR{\etalchar{+}}09]{a_IGNAROSOSZ_2009a}
M.~Iguernane, S.~Nazarov, J.-R. Roche, J.~Sokolowski, and K.~Szulc,
  \emph{Topological derivatives for semilinear elliptic equations},
  International Journal of Applied Mathematics and Computer Science \textbf{19}
  (2009), no.~2.

\bibitem[JKRS03]{a_JAKRRASO_2003a}
J.~Jarušek, M.~Krbec, M.~Rao, and J.~Sokołowski, \emph{Conical
  differentiability for evolution variational inequalities}, J. Differ.
  Equations \textbf{193} (2003), no.~1, 131--146 (English).

\bibitem[LSNS17]{a_LODSNOSO_2017a}
C.~G. Lopes, R.~B.~Dos Santos, A.~A. Novotny, and J.~Sokołowski,
  \emph{Asymptotic analysis of variational inequalities with applications to
  optimum design in elasticity}, Asymptotic Anal. \textbf{102} (2017), no.~3-4,
  227--242 (English).

\bibitem[MNP12a]{b_MANAPL_2012a}
V.~Maz'ya, S.~Nazarov, and B.~Plamenevskij, \emph{Asymptotic theory of elliptic
  boundary value problems in singularly perturbed domains: Volume i}, Operator
  Theory: Advances and Applications, Birkhäuser Basel, 2012.

\bibitem[MNP12b]{b_MANAPL_2012_b}
\bysame, \emph{Asymptotic theory of elliptic boundary value problems in
  singularly perturbed domains volume ii: Volume ii}, Operator Theory: Advances
  and Applications, Birkhäuser Basel, 2012.

\bibitem[MPS05]{a_MAPOSA_2005a}
M.~Masmoudi, J.~Pommier, and B.~Samet, \emph{The topological asymptotic
  expansion for the maxwell equations and some applications}, Inverse Problems
  \textbf{21} (2005), no.~2, 547--564.

\bibitem[NS13]{b_NOSO_2013a}
A.~A. Novotny and J.~Soko{\l}owski, \emph{Topological derivatives in shape
  optimization}, Springer Berlin Heidelberg, 2013.

\bibitem[NSZ19]{b_NOSOZO_2019a}
A.~A. Novotny, J.~Sokolowski, and A.~Zochowski, \emph{Applications of the
  topological derivative method}, Studies in Systems, Decision and Control,
  vol. 188, Springer, Cham, 2019, With a foreword by Michel Delfour.
  \MR{3887663}

\bibitem[Stu15]{a_ST_2015a}
K.~Sturm, \emph{Minimax lagrangian approach to the differentiability of
  nonlinear {PDE} constrained shape functions without saddle point assumption},
  {SIAM} Journal on Control and Optimization \textbf{53} (2015), no.~4,
  2017--2039.

\bibitem[Stu20]{a_ST_2020a}
\bysame, \emph{Topological sensitivities via a lagrangian approach for
  semilinear problems}, Nonlinearity \textbf{33} (2020), no.~9, 4310--4337.

\bibitem[SZ99]{a_SOZO_1999a}
J.~Sokolowski and A.~Zochowski, \emph{On the topological derivative in shape
  optimization}, {SIAM} Journal on Control and Optimization \textbf{37} (1999),
  no.~4, 1251--1272.

\end{thebibliography}
\bibliographystyle{amsalpha}

\end{document}